\documentclass{article}

\usepackage{cite}
\usepackage{amsmath,amssymb,amsfonts}

\usepackage{graphicx}
\usepackage{hyperref}
\hypersetup{hidelinks=true}
\usepackage{textcomp}
\usepackage{float}

\usepackage{amsthm}

\def\BibTeX{{\rm B\kern-.05em{\sc i\kern-.025em b}\kern-.08em
		T\kern-.1667em\lower.7ex\hbox{E}\kern-.125emX}}

\usepackage{graphicx, color, caption, subcaption}
\graphicspath{{fig/}}

\usepackage{pgf}
\usepackage{bm}

\newtheorem{thm}{Theorem}
\newtheorem{lem}[thm]{Lemma}
\newtheorem{prob}[thm]{Problem}

\newtheorem{cor}[thm]{Corollary}

\newcommand{\guillemets}[1]{``#1''}
\newcommand{\set}[1]{\left\{#1\right\}}
\newcommand{\case}[1]{\emph{\underline{#1}:}}

\newcommand{\myvect}[1]{\bm{#1}}
\newcommand{\myrandvect}[1]{\underline{\bm{#1}}}
\newcommand{\mymatrix}[1]{\bm{#1}}
\newcommand{\SCI}{{\mathrm{SCI}}}

\newcommand{\R}{\mathbb{R}} 
\newcommand{\E}{\mathrm{E}} 

\newcommand{\rA}{\mathcal{A}}

\newcommand{\rC}{\mathcal{C}}

\newcommand{\rE}{\mathcal{E}}

\newcommand{\rV}{\mathcal{V}}

\newcommand{\norm}[1]{\left\lVert#1\right\rVert}

\newcommand{\suchthat}{\ | \ }
\newcommand{\ie}{{i.e.},}
\newcommand{\eg}{{e.g.},}

\DeclareMathOperator*{\trace}{tr}
\DeclareMathOperator*{\minimize}{minimize}
\DeclareMathOperator*{\subject}{subject\, to:}
\DeclareUnicodeCharacter{2212}{-}

\providecommand{\keywords}[1]
{
	\small	
	\textbf{\textit{Keywords---}} #1
}

\begin{document}

\title{Optimality of Split Covariance Intersection Fusion}
\author{Colin Cros, Pierre-Olivier Amblard, Christophe Prieur
	,\\
	Jean-François Da Rocha%
		\thanks{C. Cros, P.-O. Amblard and C. Prieur are with the CNRS, Univ. Grenoble Alpes, GIPSA-lab, F-38000 Grenoble, Auvergne-Rhône-Alpes, France.
				{\tt\small colin.cros@gipsa-lab.fr, christophe.prieur@gipsa-lab.fr, 	pierre-olivier.amblard@cnrs.fr}.}
		\thanks{C. Cros and J.-F. Da Rocha are with Telespazio FRANCE, F-31100 Toulouse, Occitanie, France. {\tt\small jeanfrancois.darocha@telespazio.com}.}
	}

\maketitle

\begin{abstract}
	Linear fusion is a cornerstone of estimation theory. Optimal linear fusion was derived by Bar-Shalom and Campo in the 1980s. It requires knowledge of the cross-covariances between the errors of the estimators. In distributed or cooperative systems, these cross-covariances are difficult to compute. To avoid an underestimation of the errors when these cross-covariances are unknown, conservative fusions must be performed. A conservative fusion provides a fused estimator with a covariance bound which is guaranteed to be larger than the true (but not computable) covariance of the error. Previous research by Reinhardt \textit{et al.} proved that, if no additional assumption is made about the errors of the estimators, the minimal bound for fusing two estimators is given by a fusion called Covariance Intersection (CI). In practice, the errors of the estimators often have an uncorrelated component, because the dynamic or measurement noise is assumed to be independent. In this context, CI is no longer the optimal method and an adaptation called Split Covariance Intersection (SCI) has been designed to take advantage from these uncorrelated components. The contribution of this paper is to prove that SCI is the optimal fusion rule for two estimators under the assumption that they have an uncorrelated component. It is proved that SCI provides the optimal covariance bound with respect to any increasing cost function. To prove the result, a minimal volume that should contain all conservative bounds is derived, and the SCI bounds are proved to be the only bounds that tightly circumscribe this minimal volume.
\end{abstract}


\section{Introduction}

Fusion is one of the fundamental elements in estimation theory. It is the process of combining different pieces of information into a more accurate one. The problem of optimal fusion have been studied for decades \cite{bar1986effect}, and especially when the fused estimator is searched for as a linear combination of the estimators. The optimal linear fusion of two estimators was first proposed by Bar-Shalom and Campo \cite{bar1986effect}. They underlined the importance of the cross-covariance term. Since then, extensions for the fusion of any number of estimators have been derived, see \eg{} \cite{li2003optimal}. To perform the optimal linear fusion, the full knowledge of the second order moments of the estimators are required. In particular, the knowledge of the covariances of each estimator is not enough, but the knowledge of the cross-covariances between each pair of estimators is also necessary. This requirement may be prohibitive for some applications such as networked estimation or cooperative estimation.

Several strategies have been proposed to perform the fusion when the cross-covariances are unknown and cannot be computed. The most simple is to assume the estimators uncorrelated, then to apply the optimal scheme proposed by Bar-Shalom and Campo. However, this naive strategy may result in an underestimation of the error of estimation, see \eg{} \cite{arambel2001covariance}, and should therefore be avoided. Furthermore, as \cite{uhlmann2003covariance} points out, considering any particular possible cross-covariance to apply the optimal scheme results similarly to an underestimation of the error. As a consequence, the whole set of \emph{admissible} cross-covariances should be considered when designing the fusion. In this case, the covariance of the error of the resulting fused estimator cannot be computed as each cross-covariance would produce a different fused covariance. Instead, a \emph{conservative} bound is searched for to ensure that the error of estimation is not underestimated. This bound should be greater than the covariance of the fused estimator for all admissible cross-covariances. The optimal linear fusion problem then consists in finding a fused estimator having the smallest conservative bound. This problem can be formulated as a general optimization problem \cite{forsling2022conservative}. To fuse two estimators without any additional information on their cross-covariance, Covariance Intersection (CI) \cite{julier1997nondivergent} is proved to give the optimal bound \cite{reinhardt2015minimum}. CI provides a bound whose associated precision matrix is a convex combination of the precision matrices of both estimators. This fusion is very conservative as it does not consider any assumption on the nature of the (unknown) correlation. Hence, it encompasses the extreme case of totally correlated estimators. The set of admissible cross-covariances is the largest set possible: it contains all cross-covariances generating a positive semi-definite centralized covariance. If this set can be reduced under additional assumptions, better bounds have been derived. For example, if the estimators to be fused are known to share a common estimate, Ellipsoidal Intersection \cite{sijs2010state} have been proposed, but it is not conservative. An alternative fusion scheme called Inverse Covariance Intersection is proved to provide, in this case, a conservative bound better than CI \cite{noack2017inverse}. Another classical example is when the estimators are known to have an uncorrelated component (generally independent). In this case, an extension of CI called Split Covariance Intersection (SCI) \cite{julier2001general} also provides a better result. SCI has been applied to a great variety of problems, \eg{} SLAM \cite{julier2007using}, cooperative localization \cite{li2013cooperative}, or cooperative perception \cite{lima2021data}. There are three main reasons for the success of SCI. Firstly, it is particularly adapted for cooperative and distributed systems as the dynamical noise and the measurement noise are often assumed independent from the errors of estimation. Second, it provides a conservative bound without knowing the cross-covariance but unlike CI, it takes benefit from the independent part and therefore generates a less conservative bound. Third, it is the natural extension of CI which is known to provide the minimal bound of a fusion and achieves empirically good performances. However, to the best of our knowledge, the optimality of the SCI fusion has not yet been demonstrated.

In this paper, we focus on the linear fusion of two estimators in the situation addressed by SCI. We assume that the estimators are the sum of two components: one \emph{correlated} and the other \emph{uncorrelated}. In this context, we prove that SCI is the optimal fusion scheme in the sense that it provides the minimal bound with respect to any increasing cost function. To prove this result, we introduce a minimal volume that all conservative bounds must contain. Then, we prove that the SCI bounds are the only bounds that tightly circumscribe this volume.

The rest of the paper is organized as follows. First, Section~\ref{sec: Background} proposes an overview of the Optimal Linear Fusion Problem and of different fusion schemes. Then, the problem of optimal fusion under split covariances and our main result, the optimality of SCI, are introduced in Section~\ref{sec: Problem}. To solve this problem, Section~\ref{sec: Minimal volume} introduces and characterizes the minimal volume that must contain all conservative bounds. Section~\ref{sec: Tightness} studies the tightness around this minimal volume. Section~\ref{sec: Proof} proves our main result. A discussion is proposed in Section~\ref{sec: Discussion}. Finally, Section~\ref{sec: Conclusion} gives some perspectives.

\bigbreak
\textbf{Notation.}
In the sequel, vectors are denoted in lowercase boldface letters \eg{} $\myvect{x} \in \R^n$, and matrices in uppercase boldface variables \eg{} $\mymatrix{M} \in \R^{n\times n}$. Random variables are underlined \eg{} $\myrandvect{x}$ for a random vector. The notation $\E[\cdot]$ denotes the expected value of a random variable and $\norm{\cdot}$ the Euclidean norm of a vector. The trace, the inverse and the transpose of a matrix $\mymatrix{M}$ and the identity matrix are denoted as $\trace \mymatrix{M}$, $\mymatrix{M}^{-1}$, $\mymatrix{M}^\intercal$ and $\mymatrix{I}$ respectively. For two matrices $\mymatrix{A}$ and $\mymatrix{B}$, the notations $\mymatrix{A} \preceq \mymatrix{B}$ and $\mymatrix{A} \prec \mymatrix{B}$ mean that the difference $\mymatrix{B} - \mymatrix{A}$ is positive semi-definite and positive definite respectively. For a positive semi-definite matrix $\mymatrix{A}$, $\mymatrix{A}^{1/2}$ denotes one of its square roots. A positive definite matrix $\mymatrix{P}$ is represented in the figures by the ellipsoid $\rE(\mymatrix{P}) := \set{\myvect{x} \suchthat \myvect{x}^\intercal \mymatrix{P}^{-1}\myvect{x} \le 1}$.

\section{Background: Optimal Linear fusion}\label{sec: Background}

Consider a random state $\myrandvect{x} \in \R^n$ and two unbiased estimators $\myrandvect{\hat x}_A$ and $\myrandvect{\hat x}_B$ of $\myrandvect{x}$. The errors of estimation are denoted as $\myrandvect{\tilde x}_A = \myrandvect{\hat x}_A - \myrandvect{x}$ and $\myrandvect{\tilde x}_B = \myrandvect{\hat x}_B - \myrandvect{x}$, and their covariances and cross-covariance are denoted as $\mymatrix{C}_A = \E[\myrandvect{\tilde x}_A \myrandvect{\tilde x}_A^\intercal]$, $\mymatrix{C}_B = \E[\myrandvect{\tilde x}_B \myrandvect{\tilde x}_B^\intercal]$, and $\mymatrix{C}_{AB} = \E[\myrandvect{\tilde x}_A \myrandvect{\tilde x}_B^\intercal]$. A linear fusion consists in creating a new unbiased estimator $\myrandvect{\hat x}_F$ ($F$ for \emph{fused}) as a linear combination of $\myrandvect{\hat x}_A$ and $\myrandvect{\hat x}_B$. It depends on two gains $\mymatrix{K}_A$ and $\mymatrix{K}_B$, and is defined as:
\begin{equation}\label{eq: Fused state}
	\myrandvect{\hat x}_F(\mymatrix{K}) = \mymatrix{K}_A \myrandvect{\hat x}_A + \mymatrix{K}_B \myrandvect{\hat x}_B,
\end{equation}
where $\mymatrix{K} = (\mymatrix{K}_A, \mymatrix{K}_B)$ for compactness. When there is no ambiguity on the gains, we simply denote $\myrandvect{\hat x}_F$. The unbiasedness of $\myrandvect{\hat x}_F(\mymatrix{K})$ imposes that the gains satisfy:
\begin{equation}\label{eq: Constraint unbiasedness}\tag{$C_1$}
	\mymatrix{K}_A + \mymatrix{K}_B = \mymatrix{I}.
\end{equation}
The error of the fused estimator is defined as $\myrandvect{\tilde x}_F = \myrandvect{\hat x}_F - \myrandvect{x} = \mymatrix{K}_A\myrandvect{\tilde x}_A + \mymatrix{K}_B\myrandvect{\tilde x}_B$, and its covariance matrix is:
\begin{multline}\label{eq: MSE original}
	\mymatrix{C}_F(\mymatrix{K}) = \mymatrix{K}_A \mymatrix{C}_A \mymatrix{K}_A^\intercal + \mymatrix{K}_A \mymatrix{C}_{AB} \mymatrix{K}_B^\intercal\\
	+ \mymatrix{K}_B \mymatrix{C}_{AB}^\intercal \mymatrix{K}_A^\intercal + \mymatrix{K}_B \mymatrix{C}_B \mymatrix{K}_B^\intercal.
\end{multline}
The covariance $\mymatrix{C}_F(\mymatrix{K})$ represents the mean squared error (MSE) matrix of the fused estimator. The objective of the optimal fusion is to minimize the error of estimation, \ie{} to minimize some cost function on $\mymatrix{C}_F(\mymatrix{K})$, \eg{} its trace or its determinant.
If the covariances $\mymatrix{C}_A$ and $\mymatrix{C}_B$ and the cross-covariance $\mymatrix{C}_{AB}$ are known, the optimal fusion is well-known. It was solved by Bar-Shalom and Campo \cite{bar1986effect} and is defined as:
\begin{subequations}
	\begin{align}
		\mymatrix{C}_F^* &= \mymatrix{C}_A - (\mymatrix{C}_A - \mymatrix{C}_{AB}) \mymatrix{R}^{-1} (\mymatrix{C}_A - \mymatrix{C}_{BA}),\\
		\mymatrix{K}_A^* &= (\mymatrix{C}_B - \mymatrix{C}_{AB}^\intercal)\mymatrix{R}^{-1},\\
		\mymatrix{K}_B^* &= (\mymatrix{C}_A - \mymatrix{C}_{AB})\mymatrix{R}^{-1},
	\end{align}
\end{subequations} 
where $\mymatrix{R} = \mymatrix{C}_A + \mymatrix{C}_B - \mymatrix{C}_{AB} - \mymatrix{C}_{AB}^\intercal$. In this case, $\mymatrix{C}_F^*$ is the minimum in the Loewner ordering sense over the set of possible covariance matrices. This means that for any other gain $\mymatrix{K}$ satisfying \eqref{eq: Constraint unbiasedness}, $\mymatrix{C}_F^* \preceq \mymatrix{C}_F(\mymatrix{K})$. As a consequence, $\mymatrix{C}_F^*$ is the optimum for all increasing cost functions. As a particular case, if the estimators are uncorrelated, \ie{} $\mymatrix{C}_{AB} = \mymatrix{0}$, then $\mymatrix{C}_F^* = (\mymatrix{C}_A^{-1} + \mymatrix{C}_B^{-1})^{-1}$. The precision matrix of the optimal fused estimator is in this case the sum of the precision matrices of the two estimators, which corresponds to the classical information form of the Kalman Filter, see \eg{} \cite{anderson1979optimal}.

If the covariances $\mymatrix{C}_A$ and $\mymatrix{C}_B$ are known but the cross-covariance $\mymatrix{C}_{AB}$ is unknown (and cannot be estimated), then the MSE matrix of the fused estimator should be seen as a function of $(i)$ the gains and $(ii)$ the cross-covariance: $\mymatrix{C}_F(\mymatrix{K}, \mymatrix{C}_{AB})$. The set of cross-covariances to consider in the fusion is called the set of \emph{admissible cross-covariances} and is denoted as $\rA$. Without any additional information, this set equals to:
\begin{equation}
	\bar\rA := \set{\mymatrix{C}_{AB} \suchthat \begin{bmatrix} \mymatrix{C}_A & \mymatrix{C}_{AB} \\ {\mymatrix{C}_{AB}}^\intercal & \mymatrix{C}_B \end{bmatrix} \succeq 0 }.
\end{equation}
The set $\bar\rA$ is the most general and largest set possible, it simply imposes that $\mymatrix{C}_{AB}$ is indeed a cross-covariance matrix. With additional assumptions on the cross-covariances $\mymatrix{C}_{AB}$, the set $\rA$ is generally smaller $\bar\rA$.

As each admissible $\mymatrix{C}_{AB} \in \rA$ induces a different MSE matrix $\mymatrix{C}_F$, the \emph{true} MSE matrix of the fused estimator cannot be computed (because the \emph{true} $\mymatrix{C}_{AB}$ is unknown). Therefore, a \emph{conservative} upper-bound is searched for to ensure that the error is not underestimated. A matrix $\mymatrix{B}_F$ is said to be a \emph{conservative upper-bound} for the fusion induced by the gains $\mymatrix{K}$, if $\mymatrix{C}_F(\mymatrix{K}, \mymatrix{C}_{AB}) \preceq \mymatrix{B}_F$ for all admissible $\mymatrix{C}_{AB} \in \rA$. The problem of optimal fusion consists in finding the fusion having the optimal bound with respect to some given increasing cost function $J$:
\begin{equation}\label{eq: Optimal fusion problem}
	\left\{\begin{array}{cll}
		\minimize\limits_{\mymatrix{K}, \mymatrix{B}_F} & J(\mymatrix{B}_F) \\
		\subject{} & \mymatrix{K}_A + \mymatrix{K}_B = \mymatrix{I}\\
		& \forall \mymatrix{C}_{AB} \in \rA,\, \mymatrix{B}_F \succeq \mymatrix{C}_F(\mymatrix{K}, \mymatrix{C}_{AB})
	\end{array}\right.
\end{equation}

\begin{figure}[H]
	\begin{subfigure}{\linewidth}
		\centering
		{\input{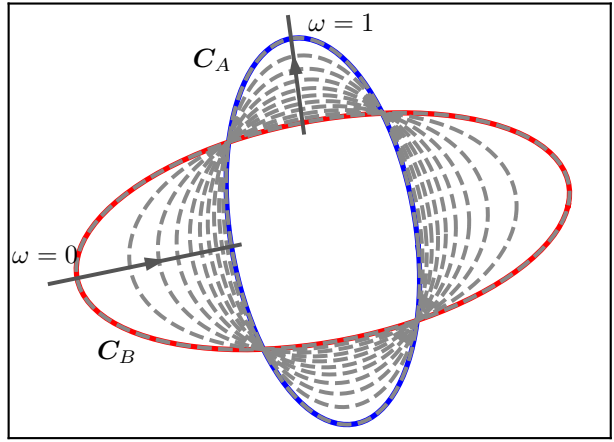}}
		\caption{CI bounds.}
		\label{sfig: all CI}
	\end{subfigure}\\
	\begin{subfigure}{\linewidth}
		\centering
		{\input{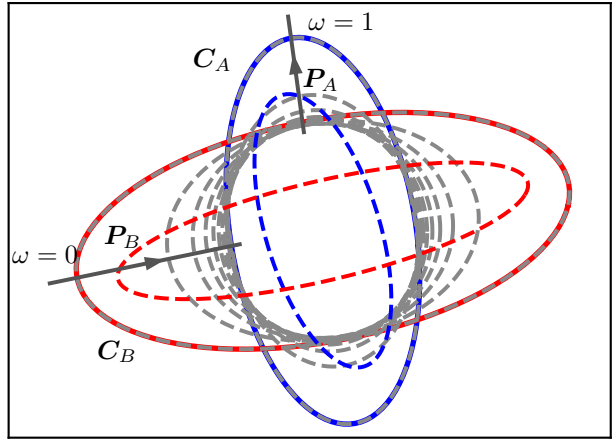}}
		\caption{SCI bounds.}
	\end{subfigure}
	\caption{Illustration of CI bounds and SCI bounds for different values of $\omega$. The bounds are represented by the grey dashed ellipses, the arrows indicate the direction $\omega$ increasing. The numerical values of the matrices are $\mymatrix{C}_A = [2,-1;-1,8]$, $\mymatrix{P}_A = [1,-1;-1,4]$, $\mymatrix{Q}_A = [1,0;0,4]$, $\mymatrix{C}_B = [13,2;2,3]$, $\mymatrix{P}_B = [9,2;2,1]$, $\mymatrix{Q}_B = [4,0;0,2]$. Unless otherwise stated, these matrices are used for all the figures. The values of $\omega$ displayed are $\omega_i = i/10$ for $i \in \set{0, \dots, 10}$.}
	\label{fig: all CI and SCI}
\end{figure}
This problem was proposed in a more general form in \cite{forsling2022conservative} to unify the formulations of different fusions. The number of variables of in \eqref{eq: Optimal fusion problem} is $O(n^2)$ where $n$ is the dimension of the state. Several candidate solutions have been proposed to solve \eqref{eq: Optimal fusion problem}, or to reduce the number of variables, for different sets $\rA$ and different cost functions $J$.
The first and most famous conservative bound, proposed by Uhlmann and Julier \cite{uhlmann1996general, julier1997nondivergent}, is CI. This fusion rule is conservative for the whole set $\bar\rA$. It is defined (in a information form) as:
\begin{subequations}\label{eq: CI equations}
	\begin{align}
		\mymatrix{B}_{\text{CI}}(\omega)^{-1}\myrandvect{\hat x}_{\text{CI}}(\omega) &= \omega \mymatrix{C}_A ^{-1} \myrandvect{\hat x}_A + (1-\omega) \mymatrix{C}_B^{-1} \myrandvect{\hat x}_B, \\
		\mymatrix{B}_{\text{CI}}^{-1}(\omega) &=\omega \mymatrix{C}_A^{-1} + (1-\omega) \mymatrix{C}_B^{-1},
	\end{align}
\end{subequations}
where $\omega$ is a parameter to choose in $[0,1]$. The precision matrix of the CI bound is a convex combination of the precision matrices of the estimators. The name Covariance Intersection comes from the fact that the ellipsoids associated with CI bounds circumscribe the intersection of the ellipsoids of the two covariances $\mymatrix{C}_A$ and $\mymatrix{C}_B$. The bounds generated by CI are illustrated in Figure~\ref{sfig: all CI}. As the parameter $\omega$ should be chosen to minimize the cost function $J$, CI still requires an optimization, but only on one variable (which can be performed efficiently). An important difference with the case where the cross-covariance $\mymatrix{C}_{AB}$ is known is that the optimal bound depends on the cost function $J$, for instance optimizing the trace or the determinant of the bound results generally in different optimal bounds. The CI fusion rule can be extended to any number of estimators by considering a convex combination of their precision matrices. The optimization of the weights becomes harder and several techniques have been developed to speed up the optimization, at a cost of suboptimality \cite{niehsen2002information, franken2005improved, wang2009fast, deng2012sequential}. The good performances of CI have first been described in \cite{chen2002estimation}: CI was proved to provide the optimal bound for the fusion of two estimators if the cost function is the trace and the whole set $\bar\rA$ is considered. More recently, this optimality was extended to any increasing cost function $J$ in \cite{reinhardt2015minimum}. Without any other assumption on the set of admissible cross-covariance matrices, any other bound is therefore either worst or non-conservative. As a consequence, CI have been applied to several estimation problems \cite{arambel2001covariance, julier2007using, guo2010covariance, lai2019cooperative}.

Considering the whole set $\bar\rA$ assumes that the errors $\myrandvect{\tilde x}_A$ and $\myrandvect{\tilde x}_B$ may be totally correlated. It is usually a very pessimistic assumption. In practice, this assumption can often be weaken and specific fusion rules have been derived for specific situations. In \cite{ajgl2019rectification}, a conservative fusion was proposed if only one component of the cross-covariance is unknown. No proof of optimality of the bound was however given. An extension to several components was then presented in \cite{ajgl2022covariance}. Another well-studied case is if the correlation between the errors $\myrandvect{\tilde x}_A$ and $\myrandvect{\tilde x}_B$ is known to be bounded. In \cite{wu2017covariance}, the authors considered a generalization of the Pearson's correlation coefficient: the largest eigenvalue of the correlation matrix: $\mymatrix{C}_A^{-1/2} \mymatrix{C}_{AB} \mymatrix{C}_B^{-1/2}$. They assumed that this correlation coefficient is bounded by some value $\rho \in [0,1]$. They proved that, in this case, the set admissible of cross-covariances $\bar\rA$ is reduced to:
\begin{equation}\label{eq: Bounded correlation set}
	\rA_{\rho} := \set{\mymatrix{C}_{AB} \suchthat \begin{bmatrix} \rho \mymatrix{C}_A & \mymatrix{C}_{AB} \\ {\mymatrix{C}_{AB}}^\intercal & \rho \mymatrix{C}_B \end{bmatrix} \succeq 0 }.
\end{equation}
The authors proposed a family of bounds for this set, and proved that the optimal bound is in the family when considering the trace as a cost function. They also proposed an algorithm to find the optimal bound within the family. In Section~\ref{sec: Discussion}, their family is proved to be in fact optimal for any increasing cost function as it is a particular case of SCI. A more general context is to assume that the estimators $\myrandvect{\hat x}_A$ and $\myrandvect{\hat x}_B$ are the sum of two components: one whose error may be correlated to an unknown degree (with the error of the other first component), and the second whose error is uncorrelated with everything else. In this situation, the covariances $\mymatrix{C}_A$ and $\mymatrix{C}_B$ are split into a correlated component, denoted as $\mymatrix{P}_A$ and $\mymatrix{P}_B$, and an uncorrelated component, denoted as $\mymatrix{Q}_A$ and $\mymatrix{Q}_B$: $\mymatrix{C}_A = \mymatrix{P}_A + \mymatrix{Q}_A$ and $\mymatrix{C}_B = \mymatrix{P}_B + \mymatrix{Q}_B$. The unknown cross-covariance $\mymatrix{C}_{AB}$ corresponds in this case to $\mymatrix{P}_{AB}$: the cross-covariance between the errors of the first components. Consequently, the set of admissible cross-covariances is reduced to:
\begin{equation}\label{eq: Split admissible set}
	\rA_{\text{Split}} := \set{\mymatrix{C}_{AB} \suchthat \begin{bmatrix} \mymatrix{P}_A & \mymatrix{C}_{AB} \\ {\mymatrix{C}_{AB}}^\intercal & \mymatrix{P}_B \end{bmatrix} \succeq 0 }.
\end{equation}
To exploit this particular situation, SCI was designed as an adaptation of CI \cite{julier2001general}. Its fusion rule is conservative for the set $\rA_{\text{Split}}$, and is also parameterized by an $\omega \in [0, 1]$ as follows:
\begin{subequations}\label{eq: SCI equations}
	\begin{multline}
		\mymatrix{B}_{\text{SCI}}(\omega)^{-1}\myrandvect{\hat x}_{\text{SCI}}(\omega) = \omega(\mymatrix{P}_A + \omega \mymatrix{Q}_A)^{-1} \myrandvect{\hat x}_A \\+ \bar\omega(\mymatrix{P}_B + \bar\omega \mymatrix{Q}_B)^{-1} \myrandvect{\hat x}_B,
	\end{multline}
	\begin{equation}
		\mymatrix{B}_{\text{SCI}}^{-1}(\omega) =\omega (\mymatrix{P}_A + \omega \mymatrix{Q}_A)^{-1} + \bar\omega (\mymatrix{P}_B + \bar\omega \mymatrix{Q}_B)^{-1},
	\end{equation}
\end{subequations}
where $\bar\omega = 1 - \omega$. SCI takes benefit from the uncorrelated components to reduce the bounds and provide a better bound than CI. A comparison between CI and SCI is proposed in Figure~\ref{fig: all CI and SCI}, it can be observed that SCI bounds are smaller. The splitting assumption of the covariances appears in a variety of contexts for the filtering of distributed and cooperative systems. For example, the integration of the SCI fusion into linear filtering is introduced in \cite{li2013split} in the context of cooperative localization. In that scenario, the independent components of the covariances come \eg{} from the dynamic noise of the agents which is assumed independent. SCI have also been applied to several other problems \cite{carrillo2013decentralized, li2013cooperative, pierre2018range, lima2021data}. Despite SCI's good performance, which makes it widely used in practice, and the fact that SCI is the natural extension of CI, SCI has not been proven to provide the optimal bound for the set $\rA_{\text{Split}}$. This is the goal of this paper.

\section{Problem formulation and main result}\label{sec: Problem}

\subsection{Optimal linear fusion with split covariances}
Consider again a random state $\myrandvect{x} \in \R^n$ and two unbiased estimators $\myrandvect{\hat x}_A$ and $\myrandvect{\hat x}_B$ of $\myrandvect{x}$. The errors of estimation are denoted as $\myrandvect{\tilde x}_A = \myrandvect{\hat x}_A - \myrandvect{x}$ and $\myrandvect{\tilde x}_B = \myrandvect{\hat x}_B - \myrandvect{x}$, and their covariances and cross-covariance are denoted as $\mymatrix{C}_A = \E[\myrandvect{\tilde x}_A \myrandvect{\tilde x}_A^\intercal]$, $\mymatrix{C}_B = \E[\myrandvect{\tilde x}_B \myrandvect{\tilde x}_B^\intercal]$, and $\mymatrix{C}_{AB} = \E[\myrandvect{\tilde x}_A \myrandvect{\tilde x}_B^\intercal]$. The two estimators are assumed to be the sum of two components:
\begin{align}
	\myrandvect{\hat x}_A &= \myrandvect{\hat x}_{A,1} + \myrandvect{\hat x}_{A,2}, &
	\myrandvect{\hat x}_B &= \myrandvect{\hat x}_{B,1} + \myrandvect{\hat x}_{B,2}.
\end{align}
The errors of the first components $\myrandvect{\tilde x}_{A,1}$ and $\myrandvect{\tilde x}_{B,1}$ are assumed to be correlated to an unknown degree, while the errors of the second components are assumed to be uncorrelated with each other and with the errors of the first components:
\begin{subequations}
	\begin{align}
		\E[\myrandvect{\tilde x}_{A,2}\myrandvect{\tilde x}_{B,2}^\intercal] &= \mymatrix{0}, \\
		\E[\myrandvect{\tilde x}_{A,2}\myrandvect{\tilde x}_{A,1}^\intercal] &= \E[\myrandvect{\tilde x}_{A,2}\myrandvect{\tilde x}_{B,1}^\intercal] = \mymatrix{0}, \\
		\E[\myrandvect{\tilde x}_{B,2}\myrandvect{\tilde x}_{A,1}^\intercal] &= \E[\myrandvect{\tilde x}_{B,2}\myrandvect{\tilde x}_{B,1}^\intercal] = \mymatrix{0}.
	\end{align}
\end{subequations}
The covariances of the errors of the first components are known and denoted as $\mymatrix{P}_A = \E[\myrandvect{\tilde x}_{A,1} \myrandvect{\tilde x}_{A,1}^\intercal]$ and $\mymatrix{P}_B = \E[\myrandvect{\tilde x}_{B,1} \myrandvect{\tilde x}_{B,1}^\intercal]$. Their cross-covariance is unknown and denoted as $\mymatrix{P}_{AB} = \E[\myrandvect{\tilde x}_{A,1} \myrandvect{\tilde x}_{B,1}^\intercal]$. The covariances of the errors of the second components are known and denoted as $\mymatrix{Q}_A = \E[\myrandvect{\tilde x}_{A,2} \myrandvect{\tilde x}_{A,2}^\intercal]$ and $\mymatrix{Q}_B = \E[\myrandvect{\tilde x}_{B,2} \myrandvect{\tilde x}_{B,2}^\intercal]$. Thus, the covariances of the errors of estimation satisfy:
\begin{subequations}
	\begin{align}
		\mymatrix{C}_A &= \mymatrix{P}_A + \mymatrix{Q}_A, &
		\mymatrix{C}_B &= \mymatrix{P}_B + \mymatrix{Q}_B, \\
		\mymatrix{C}_{AB} &= \mymatrix{P}_{AB}.
	\end{align}
\end{subequations}
The unknown cross-covariance $\mymatrix{P}_{AB}$ corresponds to the cross-covariance between $\myrandvect{\tilde x}_{A,1}$ and $\myrandvect{\tilde x}_{B,1}$. It should therefore satisfy:
\begin{equation*}
	\begin{bmatrix}
		\mymatrix{P}_A & \mymatrix{P}_{AB} \\
		\mymatrix{P}_{AB}^\intercal & \mymatrix{P}_B
	\end{bmatrix} \succeq \mymatrix{0}.
\end{equation*}
As $\mymatrix{C}_{AB} = \mymatrix{P}_{AB}$, the set of admissible cross-covariances is $\rA_\text{Split}$ introduced in \eqref{eq: Split admissible set}. The MSE of the fused estimator $\myrandvect{\hat x}_F(\mymatrix{K})$ given in \eqref{eq: MSE original} becomes:
\begin{multline}\label{eq: MSE}
	\mymatrix{C}_F(\mymatrix{K}, \mymatrix{P}_{AB}) = \mymatrix{K}_A (\mymatrix{P}_A + \mymatrix{Q}_A) \mymatrix{K}_A^\intercal + \mymatrix{K}_A \mymatrix{P}_{AB} \mymatrix{K}_B^\intercal\\
	+ \mymatrix{K}_B \mymatrix{P}_{AB}^\intercal \mymatrix{K}_A^\intercal + \mymatrix{K}_B (\mymatrix{P}_B + \mymatrix{Q}_B) \mymatrix{K}_B^\intercal.
\end{multline}
In \eqref{eq: MSE}, $\mymatrix{C}_{AB}$ has been replaced by $\mymatrix{P}_{AB}$ to emphasize that it corresponds to the cross-covariance between $\myrandvect{\tilde x}_{A,1}$ and $\myrandvect{\tilde x}_{B,1}$.
In this context, a pair $\left(\mymatrix{K}, \mymatrix{B}_F\right)$ is said to define a \emph{conservative fusion} if $\mymatrix{B}_F \succeq \mymatrix{C}_F(\mymatrix{K}, \mymatrix{P}_{AB})$ for all $\mymatrix{P}_{AB} \in \rA_{\text{Split}}$. Figure~\ref{fig: Conservative bound} presents an example of conservative fusion. In that figure, the conservativeness of the bound has been illustrated by computing the MSE $\mymatrix{C}_F(\mymatrix{K}, \mymatrix{P}_{AB})$ for particular values of $\mymatrix{P}_{AB}$ and by representing the set:
$$\rV(\mymatrix{K}) = \bigcup_{\mymatrix{P}_{AB} \in \rA_{\text{Split}}} \rE(\mymatrix{C}_F(\mymatrix{K, \mymatrix{P}_{AB}})).$$ 
The condition \guillemets{$\forall \mymatrix{P}_{AB} \in \rA_\text{Split}$, $\mymatrix{C}_F(\mymatrix{K, \mymatrix{P}_{AB}}) \preceq \mymatrix{B}_F$} is geometrically equivalent to \guillemets{$\rV(\mymatrix{K}) \subseteq \rE(\mymatrix{B}_F)$}. Intuitively, the bound $\mymatrix{B}_F$ presented on Figure~\ref{fig: Conservative bound} as a thick black line is not optimal as smaller ones exist. Throughout this paper, the covariance matrices are compared with respect to an increasing cost function $J$, increasing in the sense of the Loewner ordering \ie{} $\mymatrix{P} \preceq \mymatrix{Q} \implies J(\mymatrix{P}) \le J(\mymatrix{Q})$ (and $\mymatrix{P} \preceq \mymatrix{Q}$ with $\mymatrix{P} \ne \mymatrix{Q} \implies J(\mymatrix{P}) < J(\mymatrix{Q})$). The focus of this paper is to find the optimal linear conservative fusion of $\myrandvect{\hat x}_A$ and $\myrandvect{\hat x}_B$ for the set $\rA_{\text{Split}}$. This means solving the following problem.
\begin{prob}[Optimal Fusion with Split Covariances]\label{pro: Main problem}
	\begin{equation*}
		\left\{\begin{array}{cl}
			\minimize\limits_{\mymatrix{K}, \mymatrix{B}_F} & J(\mymatrix{B}_F) \\
			\subject{} & \mymatrix{K}_A + \mymatrix{K}_B = \mymatrix{I} \\
			& \forall \mymatrix{P}_{AB} \in \rA_\text{Split}, \, \mymatrix{B}_F \succeq \mymatrix{C}_F(\mymatrix{K}, \mymatrix{P}_{AB})
		\end{array}\right.
	\end{equation*}
	where $\mymatrix{C}_F(\mymatrix{K}, \mymatrix{P}_{AB})$ is given by \eqref{eq: MSE}.
\end{prob}

\begin{figure}[H]
	\centering
	{\input{fig/conservative_bound.pgf}}
	\caption{Illustration of a conservative fusion with gains $\mymatrix{K}_A = \mymatrix{K}_B = \mymatrix{I}/2$. The numerical values for the bound is $\mymatrix{B}_F = [6,0;0,6]$. The $5$ matrices $\mymatrix{P}_{AB}$ used to generate the ellipsoids $\rE(\mymatrix{C}_F(\mymatrix{K}, \mymatrix{P}_{AB}))$ have been set to $\mymatrix{P}_A^{1/2}\myvect{x}_i\myvect{x}_i^\intercal\mymatrix{P}_B^{1/2}$ with $\myvect{x}_i = [\cos(\pi i/5), \sin(\pi i/5)]^\intercal$.}
	\label{fig: Conservative bound}
\end{figure}

\subsection{Main result: Minimal bound}

The main result is that the bounds which are solution of Problem~\ref{pro: Main problem} are obtained through SCI \eqref{eq: SCI equations}.
\begin{thm}\label{the: Main result}
	Let $(\mymatrix{K}, \mymatrix{B}_F)$ define a conservative fusion. $(\mymatrix{K}, \mymatrix{B}_F)$ is a solution of Problem~\ref{pro: Main problem} if and only if there exists $\omega^* \in \arg\min_{\omega \in [0,1]} J(\mymatrix{B}_{\text{SCI}}(\omega))$ such that $\mymatrix{B}_F = \mymatrix{B}_{\text{SCI}}(\omega^*)$.
\end{thm}
In particular, SCI gives a particular solution to Problem~\ref{pro: Main problem}.
\begin{cor}
	A solution of Problem~\ref{pro: Main problem} is:
	\begin{subequations}\begin{align}
			\mymatrix{K}_A^* &= \omega^* \mymatrix{B}_{\text{SCI}}(\omega^*)(\mymatrix{P}_A + \omega^* \mymatrix{Q}_A)^{-1}, \\
			\mymatrix{K}_B^* &= \bar\omega^* \mymatrix{B}_{\text{SCI}}(\omega^*)(\mymatrix{P}_B + \bar\omega^* \mymatrix{Q}_B)^{-1}, \\
			\mymatrix{B}_F^* &= \mymatrix{B}_{\text{SCI}}(\omega^*),
	\end{align}\end{subequations}
	where $\omega^* \in \arg\min_{\omega \in [0,1]} J(\mymatrix{B}_{\text{SCI}}(\omega))$.
\end{cor}

Our method for proving Theorem~\ref{the: Main result} is similar to the one used in \cite{reinhardt2015minimum} to prove the optimality of CI. First, we characterize in Section~\ref{sec: Minimal volume} a minimal volume that all conservative bounds must contain. Then in Section~\ref{sec: Tightness}, this minimal volume is proved to be \emph{tightly circumscribed} by SCI bounds. The proof of Theorem~\ref{the: Main result} is finally given in Section~\ref{sec: Proof}. To help the reading, in the sequel, the proofs of the lemmas have been placed in the appendix.

\section{Minimal volume of conservative upper-bounds}\label{sec: Minimal volume}

This section introduces a minimal volume that all ellipsoids associated with conservative upper-bounds must contain.

Let $\mymatrix{K} = (\mymatrix{K}_A, \mymatrix{K}_B)$ be a pair of gains and let $\mymatrix{B}_F$ be a conservative upper-bound for the fused estimator $\myrandvect{\hat x}_F(\mymatrix{K})$. By definition, for any admissible $\mymatrix{P}_{AB} \in \rA_\text{Split}$, $\mymatrix{B}_F \succeq \mymatrix{C}_F(\mymatrix{K}, \mymatrix{P}_{AB})$. Furthermore, for a given admissible $\mymatrix{P}_{AB} \in \rA_\text{Split}$, there is a gain $\mymatrix{K}^*$ that minimizes, in the Loewner ordering, the MSE of the fused estimator. It is given by the well-known Bar-Shalom-Campo Fusion formula \cite{bar1986effect}, recalled in the following Lemma.
\begin{lem}\label{lem: Optimal gain for a cross-covariance}
	Let $\mymatrix{P}_{AB} \in \rA_\text{Split}$, for any pair of gains $\mymatrix{K}'$ satisfying \eqref{eq: Constraint unbiasedness}, it holds $\mymatrix{C}_F(\mymatrix{K}', \mymatrix{P}_{AB}) \succeq \mymatrix{C}_F(\mymatrix{K^*}, \mymatrix{P}_{AB})$, with $\mymatrix{K}^*=(\mymatrix{K}_A^*, \mymatrix{I} - \mymatrix{K}_A^*)$ defined as:
	\begin{align}
		\mymatrix{K}_A^* &= (\mymatrix{C}_B - \mymatrix{P}_{AB}^\intercal)\mymatrix{R}^{-1},\\
		\mymatrix{R} &= \mymatrix{C}_A \mymatrix{+ }\mymatrix{C}_B - \mymatrix{P}_{AB} - \mymatrix{P}_{AB}^\intercal.\label{eq: R}
	\end{align}
	We denote $\mymatrix{C}_F^*(\mymatrix{P}_{AB}) = \mymatrix{C}_F(\mymatrix{K^*}, \mymatrix{P}_{AB})$.
\end{lem}
Lemma~\ref{lem: Optimal gain for a cross-covariance} implies that:
\begin{align}\label{eq: Minimal lower-Bound (1)}
	\forall \mymatrix{P}_{AB} &\in \rA_\text{Split},& \mymatrix{B}_F &\succeq \mymatrix{C}_F(\mymatrix{K}, \mymatrix{P}_{AB}) \succeq \mymatrix{C}_F^*(\mymatrix{P}_{AB}).
\end{align}
For geometric interpretations, and as SCI bounds are defined using inverses, it is more convenient to work with precision matrices instead of covariance matrices. Let us introduce $\mymatrix{H}_F = \mymatrix{B}_F^{-1}$ and $\mymatrix{M}_F^*(\mymatrix{P}_{AB}) = \mymatrix{C}_F^*(\mymatrix{P}_{AB})^{-1}$. Eq. \eqref{eq: Minimal lower-Bound (1)} gives:
\begin{align}\label{eq: Minimal lower-Bound}
	\forall \mymatrix{P}_{AB} &\in \rA_\text{Split},& \mymatrix{H}_F &\preceq \mymatrix{M}_F^*(\mymatrix{P}_{AB}).
\end{align}

Let us define the function:
\begin{equation}
	g : \left\{\begin{array}{ccl}
		\R^n & \longrightarrow & \R\\
		\myvect{x} &\longmapsto &\min\limits_{\mymatrix{P}_{AB} \in \rA_\text{Split}} \myvect{x}^\intercal \mymatrix{M}_F^*(\mymatrix{P}_{AB}) \myvect{x}
	\end{array}\right..
\end{equation}
The minimum is well-defined as $\rA_\text{Split}$ is compact. By \eqref{eq: Minimal lower-Bound}, for all $\myvect{x} \in \R^n$, $\myvect{x}^\intercal \mymatrix{H}_F \myvect{x} \le g(\myvect{x})$. In geometric terms, this means that the ellipsoid $\rE(\mymatrix{B}_F)$ contains the union of the ellipsoids $\rE(\mymatrix{C}_F^*(\mymatrix{P}_{AB}))$ for $\mymatrix{P}_{AB} \in \rA_\text{Split}$ . Let $\rV^*$ denote this union:
\begin{equation}
	\rV^* = \bigcup_{\mymatrix{P}_{AB} \in \rA_\text{Split}}\rE(\mymatrix{C}_F^*(\mymatrix{P}_{AB})) = \set{\myvect{x} \suchthat g(\myvect{x}) \le 1}.
\end{equation}
Figure~\ref{fig: Set A^*} represents the set $\rV^*$. In that figure, Lemma~\ref{lem: Optimal gain for a cross-covariance} was illustrated with a particular matrix $\mymatrix{P}_{AB}^0 \in \rA_\text{Split}$: we note as claimed that $\rE(\mymatrix{C}_F^*(\mymatrix{P}_{AB}^{0})) \subseteq \rE(\mymatrix{C}_F(\mymatrix{K}, \mymatrix{P}_{AB}^{0}))$.
\begin{figure}[H]
	\centering
	{\input{fig/SCI_Kalman_ellipses.pgf}}
	\caption{Illustration of the set $\rV^*$. To illustrate Lemma~\ref{lem: Optimal gain for a cross-covariance}, the fusion gains were set to $\mymatrix{K}_A = \mymatrix{K}_B = \mymatrix{I}/2$ and the matrix $\mymatrix{P}_{AB}^0$ to $[2,0;-4.5,-1]$.}
	\label{fig: Set A^*}
\end{figure}

We have proven the following result.
\begin{lem}\label{lem: Minumal volume}
	If $\mymatrix{B}_F$ is an upper-bound for a fused estimator $\myrandvect{\hat x}_F$, then:
	\begin{equation}
		\rV^* \subseteq \rE(\mymatrix{B}_F).
	\end{equation}
\end{lem}
The rest of this section gives an alternative characterization of the set $\rV^*$ or equivalently, of function $g$.

For all $\omega \in [0,1]$, introduce $\mymatrix{H}_\SCI(\omega) = \mymatrix{B}_\SCI(\omega)^{-1}$ the precision matrix associated with the SCI bound, and consider $\mymatrix{H}_\SCI'(\omega)$ and $\mymatrix{H}_\SCI''(\omega)$ its first and second derivatives with respect to $\omega$. Recalling \eqref{eq: SCI equations}, they are computed as:
\begin{subequations}
	\begin{align}
		\mymatrix{H}_\SCI(\omega) &= \omega\mymatrix{A}(\omega) + \bar\omega \mymatrix{B}(\bar\omega),\\
		\mymatrix{H}_\SCI'(\omega) &= \mymatrix{A}(\omega)\mymatrix{P}_A\mymatrix{A}(\omega) - \mymatrix{B}(\omega)\mymatrix{P}_B\mymatrix{B}(\bar\omega), \\
		\mymatrix{H}_\SCI''(\omega) &= -\mymatrix{A}(\omega)[\mymatrix{Q}_A\mymatrix{A}(\omega)\mymatrix{P}_A + \mymatrix{P}_A\mymatrix{A}(\omega)\mymatrix{Q}_A]\mymatrix{A}(\omega) \notag{} \\
		&-\mymatrix{B}(\bar\omega)[\mymatrix{Q}_B\mymatrix{B}(\bar\omega)\mymatrix{P}_B+\mymatrix{P}_B\mymatrix{B}(\bar\omega)\mymatrix{Q}_B]\mymatrix{B}(\bar\omega),
	\end{align}
\end{subequations}
with:
\begin{align}
	\mymatrix{A}(\omega) &= (\mymatrix{P}_A + \omega \mymatrix{Q}_A)^{-1}, &
	\mymatrix{B}(\omega) &= (\mymatrix{P}_B + \omega \mymatrix{Q}_B)^{-1}.
\end{align}
In particular, $\mymatrix{A}(0) = \mymatrix{P}_A^{-1}$, $\mymatrix{A}(1) = \mymatrix{C}_A^{-1}$, $\mymatrix{B}(0) = \mymatrix{P}_B^{-1}$, and $\mymatrix{B}(1) = \mymatrix{C}_B^{-1}$. The following lemmas are necessary to characterize function $g$.
\begin{lem}\label{lem: Concavity}
	For all $\omega \in [0,1]$, $\mymatrix{H}_\SCI''(\omega) \prec 0$.
\end{lem}
\begin{lem}\label{lem: Correlation matrices} Let $\mymatrix{\Omega} \in \R^{n\times n}$ be a matrix.
	If $\mymatrix{\Omega}^\intercal\mymatrix{\Omega} \preceq \mymatrix{I}$, then $\mymatrix{P}_A^{1/2}\mymatrix{\Omega}\mymatrix{P}_B^{1/2} \in \rA_{\text{Split}}$.
\end{lem}

\begin{figure}[H]
	\centering
	\resizebox{\linewidth}{!}{\input{fig/SCI_tightness.pgf}}
	\caption{Illustration of Theorem~\ref{the: Expression of g}. The vector $\myvect{z}$ displayed satisfies the third case, it has been normalized so that $g(\myvect{z}) = 1$. The matrices $\mymatrix{C}_F^*(\mymatrix{P}_{AB}^*)$ and $\mymatrix{B}_\SCI(\omega_0)$ are associated with this vector: $g(\myvect{z}) = \myvect{z}^\intercal\mymatrix{M}_F^*(\mymatrix{P}_{AB}^*)\myvect{z}= \myvect{z}^\intercal\mymatrix{H}_\SCI(\omega_0)\myvect{z}$.}
	\label{fig: SCI tightness}
\end{figure}

With these notations and these two lemmas, we are in a position to characterize function $g$.
\begin{thm}\label{the: Expression of g}
	Let $\myvect{x}\in \R^n$, $\myvect{x} \ne \myvect{0}$. The three following cases are mutually exclusive and collectively exhaustive.
	\begin{enumerate}
		\item $\myvect{x}^\intercal \mymatrix{H}_\SCI'(0) \myvect{x} < 0$. In this case, $g(\myvect{x}) = \myvect{x}^\intercal \mymatrix{H}_\SCI(0) \myvect{x}$.
		\item $\myvect{x}^\intercal \mymatrix{H}_\SCI'(1) \myvect{x} > 0$. In this case, $g(\myvect{x}) = \myvect{x}^\intercal \mymatrix{H}_\SCI(1) \myvect{x}$.
		\item There exists a unique $\omega_0 \in [0,1]$ such that $\myvect{x}^\intercal \mymatrix{H}_\SCI'(\omega_0) \myvect{x} = 0$. In this case, $g(\myvect{x}) = \myvect{x}^\intercal \mymatrix{H}_\SCI(\omega_0) \myvect{x}$.
	\end{enumerate}
\end{thm}
\begin{proof}
	 Let $h : \omega \mapsto \myvect{x}^\intercal \mymatrix{H}_\SCI(\omega) \myvect{x}$. Lemma~\ref{lem: Concavity} implies that $h''(\omega) < 0$, so $h$ is strictly concave on $[0,1]$. Therefore, the three cases are mutually exclusive and collectively exhaustive.
	
	Since SCI provides conservative bounds, by Lemma~\ref{lem: Minumal volume}, for all $\myvect{x}$ and for all $\omega \in [0,1]$, $g(\myvect{x}) \ge \myvect{x}^\intercal \mymatrix{H}_\SCI(\omega)\myvect{x}$. If there exist some $\mymatrix{P}_{AB}^* \in \rA_\text{Split}$ and some $\omega_0 \in [0,1]$ such that $\myvect{x}^\intercal \mymatrix{M}_F^*(\mymatrix{P}_{AB}^*)\myvect{x} = \myvect{x}^\intercal \mymatrix{H}_\SCI(\omega_0)\myvect{x}$, then $g(\myvect{x}) \le \myvect{x}^\intercal \mymatrix{H}_\SCI(\omega_0)\myvect{x}$, so $g(\myvect{x}) = \myvect{x}^\intercal \mymatrix{H}_\SCI(\omega_0)\myvect{x}$. Therefore, let us prove for each case that there exists a $\mymatrix{P}_{AB}^* \in \rA_\text{Split}$ such that $\myvect{x}^\intercal \mymatrix{M}_F^*(\mymatrix{P}_{AB}^*)\myvect{x} = \myvect{x}^\intercal \mymatrix{H}_\SCI(\omega_0)\myvect{x}$ (with the appropriate $\omega_0$).
	
	Before starting the exhaustion of cases, notice that the covariance $\mymatrix{C}_F^*(\mymatrix{P}_{AB})$ can be expressed in the three following forms:
	\begin{subequations}\label{eq: Optimal covariances}
		\begin{align}\label{eq: Optimal covariance A}
			\mymatrix{C}_F^*(\mymatrix{P}_{AB}) &= \mymatrix{C}_A - (\mymatrix{C}_A - \mymatrix{P}_{AB})\mymatrix{R}^{-1}(\mymatrix{C}_A - \mymatrix{P}_{AB}^\intercal)\\
			&= \mymatrix{C}_B
			- (\mymatrix{C}_B - \mymatrix{P}_{AB}^\intercal)\mymatrix{R}^{-1}(\mymatrix{C}_B - \mymatrix{P}_{AB}) \label{eq: Optimal covariance B} \\
			&= \mymatrix{P}_{AB}
			+ (\mymatrix{C}_A - \mymatrix{P}_{AB})\mymatrix{R}^{-1}(\mymatrix{C}_B - \mymatrix{P}_{AB})\label{eq: Optimal covariance AB}
		\end{align}
	\end{subequations}
	where $\mymatrix{R}$ is given by \eqref{eq: R}.
	
	\case{Case 1} Assume $\myvect{x}^\intercal \mymatrix{H}_\SCI'(0) \myvect{x} < 0$, by definition of $\mymatrix{H}_\SCI'(0)$: $\myvect{x}^\intercal \mymatrix{P}_A^{-1} \myvect{x} < \myvect{x}^\intercal \mymatrix{B}(1)\mymatrix{P}_B\mymatrix{B}(1) \myvect{x}$. Then, consider the matrices:
	\begin{align*}
		\mymatrix{\Omega} &= \frac{\mymatrix{P}_A^{-1/2}\myvect{x} \myvect{x}^\intercal\mymatrix{B}(1)\mymatrix{P}_B^{1/2}}{\myvect{x}^\intercal\mymatrix{B}(1)\mymatrix{P}_B\mymatrix{B}(1)x}, & \mymatrix{P}_{AB}^* &= \mymatrix{P}_A^{1/2}\mymatrix{\Omega}\mymatrix{P}_B^{1/2}.
	\end{align*}
	The matrix $\mymatrix{\Omega}^\intercal\mymatrix{\Omega}$ has rank one: its only non-null eigenvalue $\lambda_1$ and the associated eigenvector $\myvect{z}_1$ are:
	\begin{align*}
		\lambda_1 &= \frac{\myvect{x}^\intercal\mymatrix{P}_A^{-1}\myvect{x}}{\myvect{x}^\intercal\mymatrix{B}(1)\mymatrix{P}_B\mymatrix{B}(1)\myvect{x}} < 1, & \myvect{z}_1 &= \mymatrix{P}_B^{1/2}\mymatrix{B}(1)\myvect{x}.
	\end{align*}
	Indeed by definition of $\mymatrix{\Omega}$, any vector $\myvect{y}$ orthogonal to $\myvect{z}_1$ satisfies $\mymatrix{\Omega}^\intercal\mymatrix{\Omega}\myvect{y} = \mymatrix{0}$.
	Thus, $\mymatrix{\Omega}^\intercal\mymatrix{\Omega} \preceq \mymatrix{I}$, and by Lemma~\ref{lem: Correlation matrices} the matrix $\mymatrix{P}_{AB}^* \in \rA_\text{Split}$. Let us finally prove that $\myvect{x}^\intercal \mymatrix{M}_F^*(\mymatrix{P}_{AB}^*)\myvect{x} = \myvect{x}^\intercal \mymatrix{H}_\SCI(0)\myvect{x}$. The Woodbury inverse formula applied on \eqref{eq: Optimal covariance B} gives:
	\begin{equation*}\label{eq: Optimal informations A}
		\mymatrix{M}_F^*(\mymatrix{P}_{AB}) = \mymatrix{B}(1)
		- (\mymatrix{I} - \mymatrix{B}(1)\mymatrix{P}_{AB}^\intercal)\mymatrix{R}_B^{-1}(\mymatrix{I} - \mymatrix{P}_{AB}\mymatrix{B}(1)),
	\end{equation*}
	where $\mymatrix{R}_B = \mymatrix{C}_A - \mymatrix{P}_{AB}\mymatrix{C}_B^{-1} \mymatrix{P}_{AB}^\intercal$. By construction, $\mymatrix{P}_{AB}^*\mymatrix{B}(1)\myvect{x} = \myvect{x}$, so $\myvect{x}^\intercal \mymatrix{M}_F^*(\mymatrix{P}_{AB}^*)\myvect{x} = \myvect{x}^\intercal \mymatrix{B}(1)\myvect{x}$ which concludes the proof for Case 1.
	
	\case{Case 2} The second case is symmetrical.
	
	\case{Case 3} The equality $\myvect{x}^\intercal \mymatrix{H}_\SCI'(\omega_0)\myvect{x} = 0$ develops into $\myvect{x}^\intercal\mymatrix{A}(\omega_0)\mymatrix{P}_A \mymatrix{A}(\omega_0)\myvect{x} = \myvect{x}^\intercal\mymatrix{B}(\bar\omega_0)\mymatrix{P}_B \mymatrix{B}(\bar\omega_0)\myvect{x}$. Let $\gamma = \myvect{x}^\intercal\mymatrix{A}(\omega_0)\mymatrix{P}_A \mymatrix{A}(\omega_0)\myvect{x}$ and consider the matrices:
	\begin{align*}
		 \mymatrix{\Omega} &= \frac{1}{\gamma}\mymatrix{P}_A^{1/2} \mymatrix{A}(\omega_0)\myvect{x}\myvect{x}^\intercal\mymatrix{B}(\bar\omega_0)\mymatrix{P}_B^{1/2},& \mymatrix{P}_{AB}^* &= \mymatrix{P}_A^{1/2}\mymatrix{\Omega}\mymatrix{P}_B^{1/2}.
	\end{align*}
	With the same argument as in Case 1, the matrix $\mymatrix{P}_{AB}^* \in \rA_\text{Split}$. Let us finally prove that $\myvect{x}^\intercal\mymatrix{H}_\SCI(\omega_0)\myvect{x} = \myvect{x}^\intercal\mymatrix{M}_F^*(\mymatrix{P}_{AB}^*)\myvect{x}$. Consider the product $\mymatrix{H}_\SCI(\omega_0)\mymatrix{C}_F^{**}\mymatrix{H}_\SCI(\omega_0)$ where $\mymatrix{C}_F^{**}$ abbreviates $\mymatrix{C}_F^*(\mymatrix{P}_{AB}^*)$. As $\mymatrix{H}_\SCI(\omega_0) = \omega_0\mymatrix{A}(\omega_0)+\bar\omega_0\mymatrix{B}(\bar\omega_0)$:
	\begin{multline*}
		\mymatrix{H}_\SCI(\omega_0)\mymatrix{C}_F^{**}\mymatrix{H}_\SCI(\omega_0) = \\
		\omega_0^2 \mymatrix{A}(\omega_0)\mymatrix{C}_F^{**}\mymatrix{A}(\omega_0) + \omega_0\bar\omega_0\mymatrix{A}(\omega_0)\mymatrix{C}_F^{**}\mymatrix{B}(\bar\omega_0) \\
		+\omega_0\bar\omega_0\mymatrix{B}(\bar\omega_0)\mymatrix{C}_F^{**}\mymatrix{A}(\omega_0)
		+\bar\omega_0^2\mymatrix{B}(\bar\omega_0)\mymatrix{C}_F^{**}\mymatrix{B}(\bar\omega_0).
	\end{multline*}
	Using the three expressions of $\mymatrix{C}_F^*(\mymatrix{P}_{AB})$ in \eqref{eq: Optimal covariances}, after some calculations we obtain:
	\begin{align*}
		\omega_0^2 \myvect{x}^\intercal\mymatrix{A}(\omega_0)\mymatrix{C}_F^{**}\mymatrix{A}(\omega_0)\myvect{x} &= \omega_0 \myvect{x}^\intercal\mymatrix{A}(\omega_0)\myvect{x} - \omega_0\bar\omega_0\gamma - C,\\
		\bar\omega_0^2 \myvect{x}^\intercal\mymatrix{B}(\bar\omega_0)\mymatrix{C}_F^{**}\mymatrix{B}(\bar\omega_0)\myvect{x} &= \bar\omega_0 \myvect{x}^\intercal\mymatrix{B}(\omega_0)\myvect{x} - \omega_0\bar\omega_0\gamma - C,\\
		\omega_0\bar\omega_0\myvect{x}^\intercal\mymatrix{A}(\omega_0)\mymatrix{C}_F^{**}\mymatrix{B}(\bar\omega_0)\myvect{x} &= \omega_0\bar\omega_0\gamma +C,
	\end{align*}
	with:
	\begin{multline*}
		C = \myvect{x}^\intercal(\mymatrix{I} - \bar\omega_0\mymatrix{A}(\omega_0)\mymatrix{P}_A - \omega_0\mymatrix{B}(\bar\omega_0)\mymatrix{P}_B)\mymatrix{R}^{-1}\\
		(\mymatrix{I} - \bar\omega_0\mymatrix{P}_A\mymatrix{A}(\omega_0) - \omega_0\mymatrix{P}_B\mymatrix{B}(\bar\omega_0))\myvect{x}.
	\end{multline*}
	The calculation is not detailed, however the key ideas are to note that $\omega_0 \mymatrix{A}(\omega_0)\mymatrix{C}_A = \mymatrix{I} - \bar\omega_0 \mymatrix{A}(\omega_0)\mymatrix{P}_A$ and similarly $\bar\omega_0 \mymatrix{B}(\bar\omega_0)\mymatrix{C}_B = \mymatrix{I} - \omega_0 \mymatrix{B}(\bar\omega_0)\mymatrix{P}_B$, and then to use the definition of $\mymatrix{P}_{AB}^*$ and $\gamma$.
	Thus,
	\begin{equation}\label{eq: (Pr) quadratic equality}
		\myvect{x}^\intercal\mymatrix{H}_\SCI(\omega_0)\mymatrix{C}_F^*(\mymatrix{P}_{AB}^*)\mymatrix{H}_\SCI(\omega_0)\myvect{x} = \myvect{x}^\intercal\mymatrix{H}_\SCI(\omega_0)\myvect{x}.
	\end{equation}
	Since SCI is conservative, $\mymatrix{C}_F^*(\mymatrix{P}_{AB}^*)\preceq \mymatrix{B}_\SCI(\omega_0)$, so $\mymatrix{H}_\SCI(\omega_0)\mymatrix{C}_F^*(\mymatrix{P}_{AB}^*)\mymatrix{H}_\SCI(\omega_0) \preceq \mymatrix{H}_\SCI(\omega_0)$. This inequality combined with \eqref{eq: (Pr) quadratic equality} gives:
	\begin{equation}\label{eq: (Pr) linear equality}
		\mymatrix{H}_\SCI(\omega_0)\mymatrix{C}_F^*(\mymatrix{P}_{AB}^*)\mymatrix{H}_\SCI(\omega_0)\myvect{x} = \mymatrix{H}_\SCI(\omega_0)\myvect{x}.
	\end{equation}
	Finally, by premultiplying \eqref{eq: (Pr) linear equality} by $\mymatrix{M}_F^*(\mymatrix{P}_{AB}^*)\mymatrix{B}_\SCI(\omega_0)$, $\mymatrix{H}_\SCI(\omega_0)\myvect{x} = \mymatrix{M}_F^*(\mymatrix{P}_{AB}^*)\myvect{x}$. Hence, $\myvect{x}^\intercal\mymatrix{H}_\SCI(\omega_0)\myvect{x} = \myvect{x}^\intercal\mymatrix{M}_F^*(\mymatrix{P}_{AB}^*)\myvect{x}$ which concludes the proof.
\end{proof}

The three cases of Theorem~\ref{the: Expression of g} are illustrated in Figure~\ref{fig: SCI tightness}. The condition $\myvect{x}^\intercal \mymatrix{H}_\SCI'(0) \myvect{x} < 0$ is equivalent to $\myvect{x}^\intercal \mymatrix{P}_A^{-1} \myvect{x} < \myvect{x}^\intercal \mymatrix{C}_B^{-1} \mymatrix{P}_B \mymatrix{C}_B^{-1} \myvect{x}$. As $\rE(\mymatrix{P}_A) \subset \rE(\mymatrix{C}_B\mymatrix{P}_B^{-1}\mymatrix{C}_B)$, Case 1 never occurs and $\rV^* \subset \rE(\mymatrix{H}_\SCI(0)) = \rE(\mymatrix{C}_B)$. This strict inclusion is highlighted in the zoom. Similarly, the condition $\myvect{x}^\intercal \mymatrix{H}_\SCI'(1) \myvect{x} > 0$ is equivalent to $\myvect{x}^\intercal \mymatrix{P}_B^{-1} \myvect{x} < \myvect{x}^\intercal \mymatrix{C}_A^{-1} \mymatrix{P}_A \mymatrix{C}_A^{-1} \myvect{x}$. This time, the ellipsoids $\rE(\mymatrix{P}_B)$ and $\rE(\mymatrix{C}_A\mymatrix{P}_A^{-1}\mymatrix{C}_A)$ do intersect (the intersections are highlighted by the lines). Between the intersections, Theorem~\ref{the: Expression of g} claims that $g(\myvect{x}) = \myvect{x}^\intercal\mymatrix{C}_A^{-1}\myvect{x}$ as observed. Finally, a particular vector $\myvect{z}$ realizing the third case is represented by a diamond. The ellipses associated to the matrices $\mymatrix{B}_\SCI(\omega_0)$ and $\mymatrix{C}_F^*(\mymatrix{P}_{AB}^*)$ are also plotted to illustrate the equality: $\myvect{z}^\intercal\mymatrix{H}_\SCI(\omega_0)\myvect{z} = \myvect{z}^\intercal\mymatrix{M}_F^*(\mymatrix{P}_{AB}^*)\myvect{z}$. 

Theorem~\ref{the: Expression of g} has a nice geometric interpretation. As SCI bounds are conservative, for all $\myvect{x}$ and all $\omega \in [0,1]$, $g(\myvect{x}) \ge \myvect{x}^\intercal\mymatrix{H}_\SCI(\omega)\myvect{x}$, so $g(\myvect{x}) \ge \max_{\omega}\myvect{x}^\intercal\mymatrix{H}_\SCI(\omega)\myvect{x}$. According to Theorem~\ref{the: Expression of g}, for all $\myvect{x}$, the bound $g(\myvect{x})$ is reached at some $\omega_0$. Consequently, $g$ can be re-expressed as:
\begin{align*}
	g(\myvect{x}) = \max_{\omega \in [0,1]} \myvect{x}^\intercal\mymatrix{H}_\SCI(\omega)\myvect{x}.
\end{align*}
Geometrically, the set $\rV^*$ can also be re-expressed as:
\begin{equation}
	\rV^* = \set{\myvect{x} \suchthat g(\myvect{x}) \le 1} = \bigcap_{\omega \in [0,1]} \rE(\mymatrix{B}_\SCI(\omega)).
\end{equation}
Thus, the minimal set $\rV^*$ is also characterized by the intersection of the ellipsoids induced by SCI bounds. As a consequence, $\rV^*$ is not only a volume common to all ellipsoids associated with conservative bounds, but it is also the largest volume common to these ellipsoids.

The common volume $\rV^*$ of the conservative bounds is now characterized by SCI bounds. The next section shows that SCI bounds generate also the smaller ellipsoids containing this volume.

\section{Tightness of SCI bounds over $\rV^*$}\label{sec: Tightness}

An ellipsoid $\rE(\mymatrix{P})$ is said to \emph{tightly circumscribe} $\rV^*$, if for any other ellipsoid $\rE(\mymatrix{Q})$, $\rV^* \subseteq \rE(\mymatrix{Q}) \subseteq \rE(\mymatrix{P})$ implies $\mymatrix{P} = \mymatrix{Q}$. In other words, there is no ellipsoid smaller than $\rE(\mymatrix{P})$ that contains $\rV^*$. In this section, we characterize the ellipsoids that tightly circumscribe the set $\rV^*$. The main result is the following theorem. Its proof is inspired by the proof of Kahan for the intersection of ellipsoids \cite{kahan1968circumscribing}. It is a proof by exhaustion whose cases have been adapted to SCI.

\begin{thm}\label{the: Tight inclusion}
	If $\rE(\mymatrix{B})$ tightly circumscribes $\rV^*$, then there exists $\omega_1 \in [0,1]$ such that $\mymatrix{B} = \mymatrix{B}_\SCI(\omega_1)$.
\end{thm}
\begin{proof}
	Let $\mymatrix{B}$ be a bound whose ellipsoid tightly circumscribes $\rV^*$ and denote $\mymatrix{H} = \mymatrix{B}^{-1}$ its precision matrix. Let:
	\begin{equation*}
		\phi = \min_{\myvect{x}\ne 0} \frac{g(\myvect{x})}{\myvect{x}^\intercal \mymatrix{H}\myvect{x}}.
	\end{equation*}
	By using the definition of $g$:
	\begin{equation*}
		\phi = \min_{\myvect{x}\ne 0} \min_{\mymatrix{P}_{AB}\in \rA_\text{Split}} \frac{\myvect{x}^\intercal \mymatrix{M}_F^*(\mymatrix{P}_{AB})\myvect{x}}{\myvect{x}^\intercal \mymatrix{H}\myvect{x}}.
	\end{equation*}
	As $\phi$ is the minimum over a compact (\eg{} $\myvect{x}^\intercal \mymatrix{H}\myvect{x} = 1$), it is achieved at some vector $\myvect{z}$ and some cross-covariance $\mymatrix{P}_{AB}^*$. Furthermore, by construction $\phi \ge 1$, and $\phi = 1$ otherwise, $\rE(\frac{1}{\phi}\mymatrix{B})$ would be a smaller ellipsoid than $\rE(\mymatrix{B})$ containing $\rV^*$. By Theorem~\ref{the: Expression of g}, there exists a unique $\omega_0 \in [0,1]$ such that $g(\myvect{z}) = \myvect{z}^\intercal \mymatrix{H}_\SCI(\omega_0)\myvect{z}$. Assume for the time being, that the following property is true for every $\myvect{y} \in \R^n$:
	\begin{align}\label{eq: (Pr) Inclusion SCI in bound}\tag{$\rC_y$}
		\myvect{y}^\intercal(\mymatrix{H}_\SCI(\omega_0) - \mymatrix{H})\myvect{y} \ge 0.
	\end{align}
	This is equivalent to $\mymatrix{H} \preceq \mymatrix{H}_\SCI(\omega_0)$. Then, as SCI bounds are conservative, $\rV^* \subseteq \rE(\mymatrix{B}_\SCI(\omega_0))$, and the tightness implies that $\mymatrix{B} = \mymatrix{B}_\SCI(\omega_0)$ which concludes the proof. To prove Theorem~\ref{the: Tight inclusion}, let us prove that \eqref{eq: (Pr) Inclusion SCI in bound} holds for all $\myvect{y} \in \R^n$.
	
	The following result, whose proof is also given in the appendix, is used several times in the sequel. 
	\begin{lem}\label{lem: (Pr) Main lemma}
		Let $\myvect{y} \in \R^n$, $\eta \in \R$, and define for any $\lambda \in \R$, $\myvect{x}(\lambda) = \eta \myvect{z} + \lambda \myvect{y}$. If one of the two following statements is true, 
		\begin{enumerate}
			\item $g(\myvect{x}(\lambda)) = \myvect{x}(\lambda)^\intercal\mymatrix{H}_\SCI(\omega_0) \myvect{x}(\lambda)$, for some $\lambda \ne 0$;
			\item $g(\myvect{x}(\lambda)) = \myvect{x}(\lambda)^\intercal\mymatrix{H}_\SCI(\omega_0) \myvect{x}(\lambda) + o(\lambda^2)$;
		\end{enumerate} 
		then \eqref{eq: (Pr) Inclusion SCI in bound} holds for $\myvect{y}$.
	\end{lem}
	There are three cases to consider to prove Theorem~\ref{the: Tight inclusion}, these are the same as in Theorem~\ref{the: Expression of g}.
	
	\case{Case 1} Assume $\myvect{z}^\intercal \mymatrix{H}_\SCI'(0)\myvect{z} < 0$. In this case, $\omega_0 = 0$, by Theorem~\ref{the: Expression of g}. Let $\myvect{y} \in \R^n$ be set and let us prove \eqref{eq: (Pr) Inclusion SCI in bound}. By continuity of the function $\myvect{x} \mapsto \myvect{x}^\intercal \mymatrix{H}_\SCI'(0)\myvect{x}$, if $\myvect{z}$ is slightly perturbed in the direction $\myvect{y}$, the inequality $\myvect{x}^\intercal\mymatrix{H}_\SCI'(0)\myvect{x} < 0$ still holds. Formally, there exists $\lambda > 0$ such that the vector $\myvect{x} = \myvect{z} + \lambda \myvect{y}$ also satisfies:
	\begin{equation*}
		\myvect{x}^\intercal \mymatrix{H}_\SCI'(0)\myvect{x} < 0.
	\end{equation*}
	Then, Theorem~\ref{the: Expression of g} gives $g(\myvect{x}) = \myvect{x}^\intercal \mymatrix{H}_\SCI(0)\myvect{x}$. Hence, by applying Lemma~\ref{lem: (Pr) Main lemma}, \eqref{eq: (Pr) Inclusion SCI in bound} holds. Thus, for all $\myvect{y} \in \R^n$, \eqref{eq: (Pr) Inclusion SCI in bound} holds, which concludes the proof for Case 1.
	
	\case{Case 2} Assume $\myvect{z}^\intercal \mymatrix{H}_\SCI'(1)\myvect{z} > 0$. In this case, $\omega_0 = 1$, by Theorem~\ref{the: Expression of g}. This case is symmetrical with Case 1.
	
	\case{Case 3}
	Assume $\myvect{z}^\intercal \mymatrix{H}_\SCI'(0)\myvect{z} \ge 0$ and $\myvect{z}^\intercal \mymatrix{H}_\SCI'(1)\myvect{z} \le 0$. In this case, by Theorem~\ref{the: Expression of g}, $\myvect{z}^\intercal \mymatrix{H}_\SCI'(\omega_0)\myvect{z} = 0$. There are two sub-cases to consider depending on whether $\mymatrix{H}_\SCI'(\omega_0)\myvect{z} = 0$ or not. They are the adaptations of Cases 2 and 3 in the proof of Kahan \cite{kahan1968circumscribing}.
	If $\mymatrix{H}_\SCI'(\omega_0)\myvect{z} \ne 0$, we can project \emph{almost} every $\myvect{y}$ to create a vector $\myvect{x}$ that satisfies the first assumption of Lemma~\ref{lem: (Pr) Main lemma}. If $\mymatrix{H}_\SCI'(\omega_0)\myvect{z} = 0$, we cannot, but in this case $g$ and the function $\myvect{x} \mapsto \myvect{x}^\intercal\mymatrix{H}_\SCI(\omega_0)\myvect{x}$ \emph{coincide} in a neighborhood of $\myvect{z}$ and we can apply the second case of Lemma~\ref{lem: (Pr) Main lemma}.
	
	\case{Case 3.1} Assume that $\mymatrix{H}_\SCI'(\omega_0) \myvect{z} \ne \myvect{0}$. Let $\myvect{y} \in \R^n$ such that $\myvect{y}^\intercal \mymatrix{H}_\SCI'(\omega_0)\myvect{z} \ne 0$. Define:
	\begin{equation*}
		\eta = -\frac{1}{2}\frac{\myvect{y}^\intercal \mymatrix{H}_\SCI'(\omega_0)\myvect{y}}{\myvect{y}^\intercal \mymatrix{H}_\SCI'(\omega_0)\myvect{z}},
	\end{equation*}
	so that the vector $\myvect{x} = \eta\myvect{z} + \myvect{y}$ satisfies $\myvect{x}^\intercal \mymatrix{H}_\SCI'(\omega_0)\myvect{x} = 0$. By Theorem~\ref{the: Expression of g}, $g(\myvect{x}) = \myvect{x}^\intercal \mymatrix{H}_\SCI(\omega_0)\myvect{x}$, then by Lemma~\ref{lem: (Pr) Main lemma}, \eqref{eq: (Pr) Inclusion SCI in bound} holds for $\myvect{y}$. Thus, \eqref{eq: (Pr) Inclusion SCI in bound} holds for all $\myvect{y}$ except for those on the hyperplane $\set{\myvect{y} \suchthat \myvect{y}^\intercal\mymatrix{H}_\SCI'(\omega_0)\myvect{z} = 0}$, by continuity of the function $\myvect{x} \mapsto \myvect{x}^\intercal(\mymatrix{H}_\SCI(\omega_0) - \mymatrix{H})\myvect{x}$, \eqref{eq: (Pr) Inclusion SCI in bound} holds for all $\myvect{y} \in \R^n$. This concludes the proof of Case 3.1.
	
	\case{Case 3.2} Assume that $\mymatrix{H}_\SCI'(\omega_0) \myvect{z} = \myvect{0}$. Let $\myvect{y} \in \R^n$ be set, and define for any $\lambda \in \R$ the vector $\myvect{x}(\lambda) = \myvect{z} + \lambda \myvect{y}$.
	Then, consider the function:
	\begin{equation*}
		\xi : (\lambda, \chi) \mapsto \myvect{x}(\lambda)^\intercal \mymatrix{H}_\SCI'(\chi)\myvect{x}(\lambda).
	\end{equation*}
	It is regular, satisfies $\xi(0, \omega_0) = 0$, and Lemma~\ref{lem: Concavity} gives that $\frac{\partial \xi}{\partial \chi}(0,\omega_0)= \myvect{z}^\intercal \mymatrix{H}_\SCI''(\omega_0)\myvect{z} < 0$. Then, the Implicit Function Theorem, see \eg{} \cite[Theorem 1.3.1]{krantz2002implicit}, states that there exists a continuous and differentiable function $\chi: \lambda \mapsto \chi(\lambda)$ defined on some neighborhood of $\lambda = 0$ such that for all $\lambda$ in that neighborhood:
	\begin{align*}
		\xi(\lambda, \chi(\lambda)) &= \xi(0, \omega_0) = 0, & \chi'(\lambda) = - \frac{\frac{\partial \xi}{\partial \lambda}(\lambda, \chi(\lambda))}{\frac{\partial \xi}{\partial \chi}(\lambda, \chi(\lambda))}.
	\end{align*}
	If $\chi(\lambda) \in [0,1]$, Theorem~\ref{the: Expression of g} implies that $g(\myvect{x}(\lambda)) = \myvect{x}(\lambda)^\intercal\mymatrix{H}_\SCI(\chi(\lambda))\myvect{x}(\lambda)$.
	Let us therefore consider the function:
	\begin{equation*}
		g_y : \lambda \mapsto \myvect{x}(\lambda)^\intercal \mymatrix{H}_\SCI(\chi(\lambda))\myvect{x}(\lambda).
	\end{equation*}	
	This function is twice differentiable at $0$, and using the fact that the derivative of $\chi$ at $\lambda = 0$ is:
	\begin{equation*}
		\chi'(0) = - \frac{\frac{\partial \xi}{\partial \lambda}(0, \omega_0)}{\frac{\partial \xi}{\partial \chi}(0, \omega_0)} =-\frac{2\myvect{y}^\intercal\mymatrix{H}_\SCI'(\omega)\myvect{z}}{\myvect{z}^\intercal\mymatrix{H}_\SCI''(\omega)\myvect{z}} = 0,
	\end{equation*}
	we verify that:
	\begin{align*}
		g_y(0) &= \myvect{z}^\intercal \mymatrix{H}_\SCI(\omega_0)\myvect{z}, & g_y'(0) &= 2\myvect{z}^\intercal \mymatrix{H}_\SCI(\omega_0)\myvect{y},\\
		g_y''(0) &= 2\myvect{y}^\intercal \mymatrix{H}_\SCI(\omega_0)\myvect{y}.
	\end{align*}
	Therefore, the series expansion of $g_y$ at $\lambda = 0$ is:
	\begin{align*}
		g_y(\lambda) &= (\myvect{z} + \lambda \myvect{y})^\intercal \mymatrix{H}_\SCI(\omega_0)(\myvect{z} + \lambda \myvect{y}) + o(\lambda^2),\\
		&= \myvect{x}(\lambda)^\intercal\mymatrix{H}_\SCI(\omega_0)\myvect{x}(\lambda) + o(\lambda^2). 
	\end{align*}
	Assume for the time being that there exists some $\varepsilon > 0$, such that $\forall \lambda \in [-\varepsilon, \varepsilon]$, $\chi(\lambda) \in [0,1]$. Then, by Theorem~\ref{the: Expression of g} on that neighborhood: 
	\begin{align*}
		g(\myvect{x}(\lambda)) &= \myvect{x}(\lambda)\mymatrix{H}_\SCI(\chi(\lambda))\myvect{x}(\lambda) \\
		&= \myvect{x}(\lambda)\mymatrix{H}_\SCI(\omega_0)\myvect{x}(\lambda) + o(\lambda^2).
	\end{align*}
	Thus, by Lemma~\ref{lem: (Pr) Main lemma}, \eqref{eq: (Pr) Inclusion SCI in bound} holds.
	On the other hand, for all $\varepsilon$, there exists $\lambda \in [-\varepsilon, \varepsilon]$ such that $\chi(\lambda) \notin [0,1]$, then necessarily, $\omega_0 = \chi(0) \in \set{0,1}$. Let us assume for example that $\omega_0 = 0$. In that case, for all $\varepsilon$ small enough, there exist $\lambda \in [-\varepsilon, \varepsilon]$, such that $\chi(\lambda) < 0$. Finally, note that:
	\begin{align*}
		\myvect{x}(\lambda)^\intercal \mymatrix{H}_\SCI'(0) \myvect{x}(\lambda) &= \myvect{x}(\lambda)^\intercal [\mymatrix{H}_\SCI'(0) - \mymatrix{H}_\SCI'(\chi(\lambda))] \myvect{x}(\lambda),\\
		&= - \chi(\lambda) \myvect{z}^\intercal\mymatrix{H}_\SCI''(0)\myvect{z} + o( \chi(\lambda)).
	\end{align*}
	As Lemma~\ref{lem: Concavity} implies that $\myvect{z}^\intercal\mymatrix{H}_\SCI''(0)\myvect{z} < 0$,
	for $\varepsilon$ small enough, there exists $\lambda \ne 0$ such that $\myvect{x}(\lambda)^\intercal\mymatrix{H}_\SCI'(0)\myvect{x}(\lambda) < 0$. By Theorem~\ref{the: Expression of g}, $g(\myvect{x}(\lambda)) = \myvect{x}(\lambda)^\intercal\mymatrix{H}_\SCI(0)\myvect{x}(\lambda)$, and Lemma~\ref{lem: (Pr) Main lemma} gives that \eqref{eq: (Pr) Inclusion SCI in bound} holds. The proof of Case 3 is complete.
\end{proof}
The reverse is not always true: a SCI bound may not be tight. An example has already been seen in Figure~\ref{fig: SCI tightness}: the SCI bound for $\omega = 0$ corresponding to $\mymatrix{C}_B$ is not tight. However, if $\rE(\mymatrix{B}_\SCI(\omega))$ \emph{touches} $\rV^*$, then the bound is tight as claimed in the following theorem.
\begin{thm}\label{the: Condition for tightness}
	Let $\omega_1 \in [0,1]$, $\rE(\mymatrix{B}_\SCI(\omega_1))$ tightly circumscribes $\rV^*$ if and only if there exists $\myvect{x} \ne 0$ such that $g(\myvect{x}) = \myvect{x}^\intercal\mymatrix{H}_\SCI(\omega_1)\myvect{x}$.
\end{thm}
\begin{proof}
	As SCI bounds are conservative, for all $\myvect{x}$, $\myvect{x}^\intercal\mymatrix{H}_\SCI(\omega_1)\myvect{x} \le g(\myvect{x})$. If $\forall \myvect{x} \ne 0$, $\myvect{x}^\intercal\mymatrix{H}_\SCI(\omega_1)\myvect{x} < g(\myvect{x})$, consider $\phi = \min_{\myvect{x} \ne 0} g(\myvect{x})/\myvect{x}^\intercal\mymatrix{H}_\SCI(\omega_1)\myvect{x}$. As $\phi$ is the minimum over a compact, it is reached at some vector $\myvect{z} \ne 0$. Since $\myvect{z}^\intercal\mymatrix{H}_\SCI(\omega)\myvect{z} < g(\myvect{z})$, $\phi > 1$. Then, it can be verified that $\rV^* \subseteq \rE(\frac{1}{\phi}\mymatrix{B}_\SCI(\omega_1)) \subseteq \rE(\mymatrix{B}_\SCI(\omega_1))$. Thus, $\rE(\mymatrix{B}_\SCI(\omega_1))$ does not tightly circumscribe $\rV^*$.
	
	Conversely, assume that there exists $\myvect{x} \ne 0$ such that $g(\myvect{x}) = \myvect{x}^\intercal\mymatrix{H}_\SCI(\omega_1)\myvect{x}$ and consider another ellipsoid $\rE(\mymatrix{B})$ such that $\rV^* \subseteq \rE(\mymatrix{B}) \subseteq \rE(\mymatrix{B}_\SCI(\omega_1))$. We can chose $\rE(\mymatrix{B})$ tight and Theorem~\ref{the: Tight inclusion} states that there exists $\omega_2 \in [0,1]$ such that $\mymatrix{B} = \mymatrix{B}_\SCI(\omega_2)$. Let us prove that $\omega_1 = \omega_2$. By assumption $g(\myvect{x})=\myvect{x}^\intercal \mymatrix{H}_\SCI(\omega_1)\myvect{x}$ and $\myvect{x}^\intercal \mymatrix{H}_\SCI(\omega_1)\myvect{x} \le \myvect{x}^\intercal \mymatrix{H}_\SCI(\omega_2)\myvect{x} \le g(\myvect{x})$, so $g(\myvect{x})=\myvect{x}^\intercal \mymatrix{H}_\SCI(\omega_2)\myvect{x}$. As the function $h : \omega \mapsto \myvect{x}^\intercal \mymatrix{H}_\SCI(\omega)\myvect{x}$ is strictly concave on $[0,1]$ by Lemma~\ref{lem: Concavity}, it reaches its maximum exactly once on $[0,1]$. Furthermore, as SCI is conservative, $\forall \omega \in[0,1]$, $h(\omega) \le g(\myvect{x}) = h(\omega_1) = h(\omega_2)$. Thus, $\omega_1 = \omega_2$, and $\rE(\mymatrix{B}_\SCI(\omega_1))$ tightly circumscribes $\rV^*$.
\end{proof}

\section{Proof of the main result}\label{sec: Proof}

Thanks to Theorem~\ref{the: Tight inclusion}, we are now in a position to prove our main result.

\begin{proof}[Proof of Theorem~\ref{the: Main result}]
	Let $\omega^* \in \arg\min_{\omega \in [0,1]} J(\mymatrix{B}_\SCI(\omega))$. Let us prove that for any conservative bound $\mymatrix{B}_F$, $J(\mymatrix{B}_\SCI(\omega^*)) \le J(\mymatrix{B}_F)$.
	
	Let a pair $(\mymatrix{K}, \mymatrix{B}_F)$ define a conservative fusion. According to Lemma~\ref{lem: Minumal volume}, the ellipsoid $\rE(\mymatrix{B}_F)$ contains $\rV^*$.
	If the inclusion is tight, according to Theorem~\ref{the: Tight inclusion}, there exists $\omega_1 \in [0,1]$ such that $\mymatrix{B}_\SCI(\omega_1) = \mymatrix{B}_F$. In this case, $J(\mymatrix{B}_F) = J(\mymatrix{B}_\SCI(\omega_1))$. If the inclusion is not tight, there exists a smaller ellipsoid which circumscribes tightly $\rV^*$ and then there exists $\omega_1 \in [0,1]$ such that $\mymatrix{B}_\SCI(\omega_1) \preceq \mymatrix{B}_F$ and $\mymatrix{B}_\SCI(\omega_1) \ne \mymatrix{B}_F$. In this case, since $J$ is increasing, $J(\mymatrix{B}_\SCI(\omega_1)) < J(\mymatrix{B}_F)$. In both cases as $\omega^* \in \arg\min_{\omega \in [0,1]} J(\mymatrix{B}_\SCI(\omega))$, $J(\mymatrix{B}_\SCI(\omega^*)) \le J(\mymatrix{B}_F)$. Hence $\mymatrix{B}_\SCI(\omega^*)$ reaches the minimum.
	
	Furthermore, $(\mymatrix{K}, \mymatrix{B}_F)$ is also solution of Problem~\ref{pro: Main problem} if and only if $J(\mymatrix{B}_\SCI(\omega^*)) = J(\mymatrix{B}_\SCI(\omega_1)) = J(\mymatrix{B}_F)$. By definition, the first equality is equivalent to $\omega_1 \in \arg\min_{\omega \in [0,1]} J(\mymatrix{B}_\SCI(\omega))$, and as $\mymatrix{B}_\SCI(\omega_1) \preceq \mymatrix{B}_F$, the second equality is equivalent to $\mymatrix{B}_\SCI(\omega_1) = \mymatrix{B}_F$.
\end{proof}

\section{Discussion}\label{sec: Discussion}

The most important implication of Theorem~\ref{the: Main result} is the drastic simplification of Problem~\ref{pro: Main problem}. It can be reformulated as:
\begin{prob}[New Optimal Fusion with Split Covariances]
	\begin{equation*}
		\left\{\begin{array}{cll}
			\minimize\limits_{\omega} & J(\mymatrix{B}_\SCI(\omega)) \\
			\subject{} & 0 \le \omega \le 1
		\end{array}\right.
	\end{equation*}
	where $\mymatrix{B}_\SCI(\omega)$ is given by \eqref{eq: SCI equations}.
\end{prob}
As a consequence, instead of optimizing for $O(n^2)$ unknowns, there is now only one unknown lying on a segment. Such optimization becomes trivial for modern solvers. Even a linear search would be efficient. In addition, the function $\omega \mapsto J(\mymatrix{B}_\SCI(\omega))$ has been proven to be convex if the cost function is the determinant \cite{li2021woptimization}, making the optimization process even faster. It is also convex for the trace as stated in the following lemma.
\begin{lem}\label{lem: Convexity Trace}
	The function $\omega \mapsto \trace(\mymatrix{B}_\SCI(\omega))$ is convex on $[0,1]$.
\end{lem}
For low dimensions, $n \le 4$, \cite{reinhardt2012closed} proposes closed-formed solutions for the parameters $\omega$ minimizing the trace and the determinant of the CI bounds. These solutions may be adapted with SCI to speed-up even more the optimization.

Geometrically, the optimality of the SCI fusion have been proved using the tightness over the minimal set $\rV^*$. Therefore, the optimality of the SCI fusion does not depend on the (increasing) cost function. However, similarly to the CI fusion, the optimal bound generally depends on the cost function: \eg{} optimizing the trace or the determinant results generally in different optimal bounds. A notable difference with CI is that, except in the trivial cases where $\mymatrix{C}_A \preceq \mymatrix{C}_B$ or $\mymatrix{C}_B \preceq \mymatrix{C}_A$, all CI bounds tightly circumscribe the minimal volume \cite{kahan1968circumscribing} (in the case of CI, it is the intersection of the ellipsoids $\rE(\mymatrix{C}_A)$ and $\rE(\mymatrix{C}_B)$). This implies that all CI bounds reach the minimal bound for some cost function. For SCI, this is not the case as stated in Theorem~\ref{the: Condition for tightness}. For example, if $\mymatrix{P}_A = \mymatrix{P}_B = \mymatrix{Q}_A = \mymatrix{Q}_B = \mymatrix{I}$, it can be proved (it is not detailed here) that there is only one tight SCI bound: for $\omega = 1/2$. In this case, $\mymatrix{B}_\SCI(1/2)$ is the minimal bound for all increasing cost functions.

SCI provides the optimal fusion bound when the estimators are split into a correlated (but to an unknown degree) component and an uncorrelated component. Therefore, by considering the limit cases in which one of these two components is null, we rediscover two well-known optimal fusions. If there is no uncorrelated component, \ie{} if $\mymatrix{Q}_A = \mymatrix{Q}_B = \mymatrix{0}$, then the SCI equations \eqref{eq: SCI equations} become the CI equations \eqref{eq: CI equations}, and the admissible set $\rA_{\text{Split}}$ becomes $\bar\rA$. CI has indeed been proven to provide the optimal bound for the set $\bar\rA$ \cite{reinhardt2015minimum}; in fact, this was the original motivation of this work. On the other hand, if the correlated components are null, \ie{} if $\mymatrix{P}_A = \mymatrix{P}_B = \mymatrix{0}$, the only admissible cross-covariance is $\mymatrix{P}_{AB} = \mymatrix{0}$ and the admissible set $\rA$ is a singleton. In this case, the SCI equations \eqref{eq: SCI equations} becomes the well-known Information Filter (IF):
\begin{subequations}\label{eq: Information fusion equations}
	\begin{align}
		\mymatrix{C}_{\text{IF}}^{-1} \myrandvect{\hat x}_{\text{IF}} &= \mymatrix{C}_A^{-1} \myrandvect{\hat x}_A + \mymatrix{C}_B^{-1} \myrandvect{\hat x}_B,\\
		\mymatrix{C}_{\text{IF}}^{-1} &= \mymatrix{C}_A^{-1} + \mymatrix{C}_B^{-1}.
 	\end{align}
\end{subequations}
This corresponds to the application of the Bar-Shalom-Campo formula with $\mymatrix{C}_{AB} = \mymatrix{0}$. The bound $\mymatrix{C}_{\text{IF}}$ is actually the \emph{true} covariance of the error of the fused estimator. SCI is therefore a more general optimal fusion rule than both CI and the information filter.

\begin{figure}
	\centering
	{\input{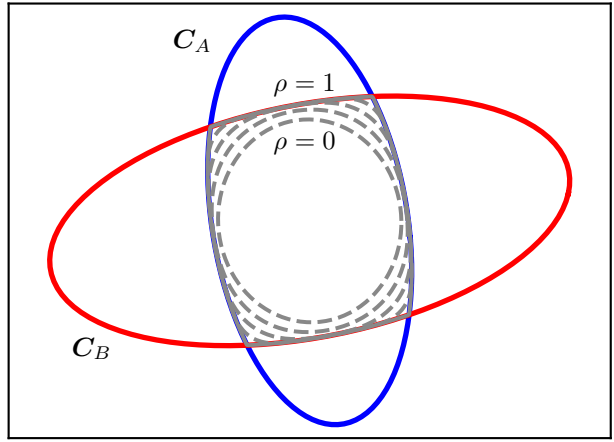}}
	\caption{Evolution of the set $\rV^*$ as function of the maximal admissible correlation coefficient $\rho$. The dashed curves represent the border of the set $\rV^*$ obtained for different correlation coefficients. The matrices $\mymatrix{C}_A$ and $\mymatrix{C}_B$ have been set to the same values as in Figure~\ref{fig: all CI and SCI}, and the correlation coefficients to $\rho = 0.25i$ for $i\in \set{0, \dots, 4}$. In this context: $\mymatrix{P}_A = \rho\mymatrix{C}_A$, $\mymatrix{P}_B = \rho\mymatrix{C}_B$, $\mymatrix{Q}_A = (1-\rho)\mymatrix{C}_A$, and $\mymatrix{Q}_B = (1-\rho)\mymatrix{C}_B$.}
	\label{fig: Evolution of rA}
\end{figure}

Since SCI has been proved to be optimal, it can be applied to the problem of partially correlated estimates treated in \cite{wu2017covariance}. This problem focuses on the case where the estimators are known to have a \emph{bounded} correlation, in the sense that the maximum eigenvalue of their correlation matrix is bounded by some parameter $\rho \in [0,1]$. The authors of \cite{wu2017covariance} proved that this assumption is equivalent to consider the set of admissible cross-covariances $\rA_{\rho}$ defined in \eqref{eq: Bounded correlation set}. The set $\rA_{\rho}$ is a particular case of $\rA_{\text{Split}}$ where the matrices are set to $\mymatrix{P}_A = \rho\mymatrix{C}_A$, $\mymatrix{P}_B = \rho\mymatrix{C}_B$, $\mymatrix{Q}_A = (1-\rho)\mymatrix{C}_A$ and $\mymatrix{Q}_B = (1-\rho)\mymatrix{C}_B$. Therefore, Theorem~\ref{the: Main result} states that the optimal fusion is of the form (after simplification):
\begin{subequations}\label{eq: Particular form Wu}
	\begin{align}
		\mymatrix{K}_A &= \frac{\omega}{\rho + \omega(1-\rho)} \mymatrix{B}_F\mymatrix{C}_A^{-1}, \\
		\mymatrix{K}_B &= \frac{\bar\omega}{\rho + \bar\omega(1-\rho)} \mymatrix{B}_F \mymatrix{C}_B^{-1}, \\
		\mymatrix{B}_F^{-1} &= \frac{\omega}{\rho + \omega(1-\rho)}\mymatrix{C}_A^{-1} + \frac{\bar\omega}{\rho + \bar\omega(1-\rho)}\mymatrix{C}_B^{-1},
	\end{align}
\end{subequations}
with $\omega \in [0,1]$. By defining $\gamma = \frac{1-\omega}{\omega}$, or equivalently letting $\omega = \frac{1}{1 + \gamma}$, \eqref{eq: Particular form Wu} becomes:
 \begin{subequations}\label{eq: Particular form Wu (2)}
 	\begin{align}
 		\mymatrix{K}_A &= (1+\gamma\rho)^{-1} \mymatrix{B}_F\mymatrix{C}_A^{-1}, \\
 		\mymatrix{K}_B &= (1+\gamma^{-1}\rho)^{-1} \mymatrix{B}_F \mymatrix{C}_B^{-1}, \\
 		\mymatrix{B}_F^{-1} &= (1+\gamma\rho)^{-1}\mymatrix{C}_A^{-1} + (1+\gamma^{-1}\rho)^{-1}\mymatrix{C}_B^{-1}.
 	\end{align}
 \end{subequations}
The bounds defined by \eqref{eq: Particular form Wu (2)} was proposed in \cite{wu2017covariance} and proved to provide the minimal bound when considering the trace as a cost function. Theorem~\ref{the: Main result} now adds that \eqref{eq: Particular form Wu (2)} also provides the minimal bound when considering any increasing cost function.
Figure~\ref{fig: Evolution of rA} presents the evolution of the minimum volume $\rV^*$ as the maximum correlation $\rho$ ranges from $0$ to $1$. As it can been seen, if $\rho = 0$, then $\rV^*$ is an ellipse: it is $\rE(\mymatrix{C}_{\text{IF}})$. On the other hand if $\rho = 1$, then $\rV^*$ corresponds to the intersection of $\rE(\mymatrix{C}_A)$ and $\rE(\mymatrix{C}_B)$ as in CI. As the maximum correlation $\rho$ increases, the set $\rV^*$ also increases what illustrates why the information filter gives smaller bound than SCI and than CI: the less information, the greater the minimal volume $\rV^*$, and the larger the bound.

\section{Conclusion}\label{sec: Conclusion}

The SCI fusion has been proven to reach the minimum covariance bound for the fusion of two estimators having split covariances. SCI can be adapted to any number of estimators by extending the convex combination of precision matrices. In this case, it still defines a conservative fusion \cite{julier2001general}. Unfortunately, it has been demonstrated, in \eg{} \cite{reinhardt2015minimum, ajgl2018fusion}, that CI is not optimal for more than two estimators. The same arguments also apply to SCI. 
As a consequence, SCI and CI are both suboptimal when applied with sequential inputs. The definition of recursive filters, such as the well-known Kalman filter, is therefore guaranteed to provided suboptimal bounds. Future work should focus on the research of the optimal fusion for more than two estimators. This challenging task is related to the problem of circumscribing an ellipsoid to the intersection of several ellipsoids. This problem is open since at least the work of Kahan \cite{kahan1968circumscribing}.

\appendix
\section{Proofs of the lemmas}\label{ap: Proofs of the lemmas}
\subsection{Proof of Lemma~\ref{lem: Optimal gain for a cross-covariance}}

Consider an unbiased fusion, \ie{} let $\mymatrix{K}_B = \mymatrix{I} - \mymatrix{K}_A$. Then, from \eqref{eq: MSE}:
\begin{multline*}
	\mymatrix{C}_F(\mymatrix{K}, \mymatrix{P}_{AB}) - \mymatrix{C}_F^*(\mymatrix{P}_{AB}) =\\
	\left(\mymatrix{K}_A - (\mymatrix{C}_B - \mymatrix{P}_{AB}^\intercal)\mymatrix{R}^{-1}\right)
	\\
	\times \mymatrix{R} \left(\mymatrix{K}_A^\intercal - \mymatrix{R}^{-1}(\mymatrix{C}_B - \mymatrix{P}_{AB})\right) \succeq \mymatrix{0}.
\end{multline*}

\subsection{Proof of Lemma~\ref{lem: Concavity}}
First, note that for every positive definite matrices $\mymatrix{P}$ and $\mymatrix{Q}$: $\mymatrix{P}(\mymatrix{P} + \mymatrix{Q})^{-1}\mymatrix{Q} = (\mymatrix{P}^{-1} + \mymatrix{Q}^{-1})^{-1} \succ \mymatrix{0}$. The result then follows by applying this result with $\mymatrix{P} = \mymatrix{P}_A$ and $\mymatrix{Q} = \omega \mymatrix{Q}_A$, and $\mymatrix{P} = \mymatrix{P}_B$ and $\mymatrix{Q} = \bar\omega \mymatrix{Q}_B$.

\subsection{Proof of Lemma~\ref{lem: Correlation matrices}}

Consider $\mymatrix{\Omega}$ such that $\mymatrix{\Omega}^\intercal\mymatrix{\Omega} \preceq \mymatrix{I}$ and let $\mymatrix{P}_{AB} = \mymatrix{P}_A^{1/2} \mymatrix{\Omega} \mymatrix{P}_B^{1/2}$. Let us prove that $\mymatrix{P}_{AB} \in \rA_{\text{Split}}$. By construction:
\begin{equation*}
	\begin{bmatrix}
		\mymatrix{P}_A & \mymatrix{P}_{AB} \\
		\mymatrix{P}_{AB}^\intercal & \mymatrix{P}_B
	\end{bmatrix}
	= \begin{bmatrix}
		\mymatrix{P}_A^{1/2} & \mymatrix{0} \\
		\mymatrix{0} & \mymatrix{P}_B^{1/2}
	\end{bmatrix}
	\begin{bmatrix}
		\mymatrix{I} & \mymatrix{\Omega} \\
		\mymatrix{\Omega}^\intercal & \mymatrix{I}
	\end{bmatrix}
	\begin{bmatrix}
		\mymatrix{P}_A^{1/2} & \mymatrix{0} \\
		\mymatrix{0} & \mymatrix{P}_B^{1/2}
	\end{bmatrix}.
\end{equation*}
As $\mymatrix{\Omega}^\intercal\mymatrix{\Omega} \preceq \mymatrix{I}$, the matrix $\begin{bmatrix}
	\mymatrix{I} & \mymatrix{\Omega} \\
	\mymatrix{\Omega}^\intercal & \mymatrix{I}
\end{bmatrix}$ is positive semi-definite, see \eg{} \cite[Lemma 7.7.6]{horn2012matrix}, and thus $\mymatrix{P}_{AB} \in \rA_{\text{Split}}$.

\subsection{Proof of Lemma~\ref{lem: (Pr) Main lemma}}

Let us first prove that $(\mymatrix{H}_\SCI(\omega_0) - \mymatrix{H})\myvect{z} = \myvect{0}$.
By definition of $\myvect{z}$ and $\mymatrix{M}_F^*(\mymatrix{P}_{AB}^*)$:
\begin{align*}
	\myvect{z}^\intercal (\mymatrix{H} - \mymatrix{M}_F^*(\mymatrix{P}_{AB}^*)) \myvect{z} &= 0, \\
	\myvect{z}^\intercal (\mymatrix{H}_\SCI(\omega_0) - \mymatrix{M}_F^*(\mymatrix{P}_{AB}^*)) \myvect{z} &= 0.
\end{align*}
As both $\mymatrix{B}$ and $\mymatrix{B}_\SCI(\omega_0)$ are conservative, $\mymatrix{H} - \mymatrix{M}_F^*(\mymatrix{P}_{AB}^*) \preceq \mymatrix{0} $ and $\mymatrix{H}_\SCI(\omega_0) - \mymatrix{M}_F^*(\mymatrix{P}_{AB}^*) \preceq \mymatrix{0}$. Therefore $\mymatrix{H}\myvect{z} = \mymatrix{M}_F^*(\mymatrix{P}_{AB}^*)\myvect{z} = \mymatrix{H}_\SCI(\omega_0)\myvect{z}$. Hence $(\mymatrix{H}_\SCI(\omega_0) - \mymatrix{H})\myvect{z} = \myvect{0}$.

Assume now that for some $\eta \in \R$ and some $\lambda \ne 0$, the vector $\myvect{x} = \eta \myvect{z} + \lambda \myvect{y}$ satisfies $g(\myvect{x}) = \myvect{x}^\intercal\mymatrix{H}_\SCI(\omega_0)\myvect{x}$. As $\mymatrix{H}$ is conservative, $g(\myvect{x}) \ge \myvect{x}^\intercal\mymatrix{H}\myvect{x}$. Expanding $\myvect{x}^\intercal (\mymatrix{H}_\SCI(\omega_0) - \mymatrix{H})\myvect{x}$ gives:
\begin{equation*}
	0 \le \myvect{x}^\intercal (\mymatrix{H}_\SCI(\omega_0) - \mymatrix{H})\myvect{x} = \lambda^2\myvect{y}^\intercal (\mymatrix{H}_\SCI(\omega_0) - \mymatrix{H})\myvect{y}.
\end{equation*}
Thus, $\myvect{y}^\intercal (\mymatrix{H}_\SCI(\omega_0) - \mymatrix{H})\myvect{y} \ge 0$.

Similarly, if $g(\myvect{x}(\lambda)) = \myvect{x}(\lambda)^\intercal\mymatrix{H}_\SCI(\omega_0)\myvect{x}(\lambda) + o(\lambda^2)$, then:
\begin{equation*}
	\myvect{y}^\intercal(\mymatrix{H}_\SCI(\omega_0) - \mymatrix{H})\myvect{y} = \frac{g(\myvect{x}(\lambda)) - \myvect{x}(\lambda)^\intercal\mymatrix{H}\myvect{x}(\lambda)}{\lambda^2} + o(1).
\end{equation*}
As $\forall \lambda$, $g(\myvect{x}(\lambda)) - \myvect{x}(\lambda)^\intercal\mymatrix{H}\myvect{x}(\lambda) \ge 0$:
\begin{equation*}
	\myvect{y}^\intercal(\mymatrix{H}_\SCI(\omega_0) - \mymatrix{H})\myvect{y} \ge o(1).
\end{equation*}
Thus, as the left-hand side is constant, by considering the limit when $\lambda$ goes to $0$:
$\myvect{y}^\intercal (\mymatrix{H}_\SCI(\omega_0) - \mymatrix{H})\myvect{y} \ge 0$ as claimed.

\subsection{Proof of Lemma~\ref{lem: Convexity Trace}}

Recall that we note $\mymatrix{H}_\SCI(\omega) = \mymatrix{B}_\SCI(\omega)^{-1}$. Let $f : \omega \mapsto \trace(\mymatrix{B}_\SCI(\omega))$.

The second derivative of $\omega \mapsto \mymatrix{B}_\SCI(\omega)$ is:
\begin{multline*}
	\mymatrix{B}_\SCI''(\omega) = 2 \mymatrix{B}_\SCI(\omega) \mymatrix{H}_\SCI'(\omega) \mymatrix{B}_\SCI(\omega) \mymatrix{H}_\SCI'(\omega) \mymatrix{B}_\SCI(\omega) \\
	- \mymatrix{B}_\SCI(\omega) \mymatrix{H}_\SCI''(\omega)\mymatrix{B}_\SCI(\omega).
\end{multline*}
As $\forall \omega \in [0,1]$, $\mymatrix{B}_\SCI(\omega)$ is positive definite and $\mymatrix{H}_\SCI''(\omega)$ is negative definite by Lemma~\ref{lem: Concavity}, $\mymatrix{B}_\SCI''(\omega)$ is positive definite.

Thus, on $[0,1]$ the second derivative of $f$ satisfies:
\begin{equation*}
	f''(\omega) = \trace(\mymatrix{B}_\SCI''(\omega)) > 0.
\end{equation*}
Hence, $f$ is convex.


\bibliographystyle{plain} 
\bibliography{references}


\end{document}